\documentclass[a4paper]{amsart}
\usepackage{latexsym,graphicx}
\usepackage[parfill]{parskip}
\usepackage{amssymb}
\usepackage{amsthm}
\usepackage{amstext}
\usepackage{amsmath}
\usepackage{enumerate}
\usepackage{enumitem}
\usepackage{stmaryrd}
\usepackage[figuresleft]{rotating}
\usepackage{ mathrsfs }
\usepackage{hyperref}
\usepackage{comment}
\usepackage[active]{srcltx}
\usepackage[all]{xy}
\usepackage{mathtools}
\usepackage[a4paper]{geometry}
\usepackage{pict2e}
\usepackage{xcolor}
\usepackage{setspace}

\DeclareMathOperator\Hom{Hom}

\DeclareMathOperator\im{im}

\DeclareMathOperator\id{id}

\DeclareMathOperator\dom{dom}
\DeclareMathOperator\ind{Ind}
\DeclareMathOperator\res{Res}
\DeclareMathOperator\groth{Groth}
\DeclareMathOperator\ann{Ann}
\DeclareMathOperator\len{len}
\newtheorem{mythm}{Theorem}
\newtheorem{myex}[mythm]{Example}
\newtheorem{mylem}[mythm]{Lemma}
\newtheorem{myprop}[mythm]{Proposition}
\newtheorem{mycor}[mythm]{Corollary}

\numberwithin{equation}{section}
\newcommand{\Z}{\mathbb{Z}}

\newcommand{\C}{\mathbb{C}}
\newcommand{\N}{\mathbb{N}}

\newcommand{\dmod}{D_{2n}\text{-Mod}}
\newcommand{\mmod}{\text{-Mod}}

\newcommand{\lb}{\llbracket}
\newcommand{\rb}{\rrbracket}

\newcommand{\rp}{\rrparenthesis}

\begin{document}
\title[Induction and restriction on representations of dihedral groups]
{Induction and restriction\\ on representations of dihedral groups}
\author[B.~Dubsky]{Brendan Dubsky}

\begin{abstract}
We study the algebras generated by restriction and induction operations on complex modules over dihedral groups. In the case where the orders of all dihedral groups involved are not divisible by four, we describe the relations, a basis, the center, and a decomposition into indecomposables of these algebras. 
\end{abstract}

\maketitle

\section{Introduction}\label{s1}
The present paper seeks to investigate the relations between compositions of restriction and induction functors when applied to modules over dihedral groups, and in particular to study certain algebras $A_{P,\mathcal{M}}$ derived from these relations. 

\subsection{Motivation}
Let $S_n$ be the symmetric group on $n$ elements, and $S_n$-mod the category of its finitely generated modules over $\C$. Then we have the usual induction functor $\ind_n:S_n\text{-mod}\rightarrow S_{n+1}\text{-mod}$ and restriction functor $\res_n:S_{n+1}\text{-mod}\rightarrow S_{n}\text{-mod}$. Consider now the direct sum of all of the categories $S_n$-mod, as $n$ ranges over the nonnegative integers. This category comes equipped with two exact endofunctors $\ind$ and $\res$ obtained by adding up all $\ind_n$ and $\res_n$ respectively. By taking the Grothendieck group of the whole construction, we obtain a vector space with two linear operators $[\res]$ and $[\ind]$ respectively. The classical Branching rule for the symmetric group (cf., e.g.,\cite[p. 77]{Sa01}) implies that these two linear operators satisfy the defining relations of the Heisenberg algebra (as is used for instance in \cite{Kh14}), namely 
\begin{equation*}
[\res][\ind]-[\ind][\res]=[\id].
\end{equation*}
Moreover, it is known that the above equality admits the functorial ``upgrade''
\begin{equation*}
\res\circ\ind\cong\id+\ind\circ\res.
\end{equation*}

From the study of the decomposition numbers of certain Hecke algebras emerged a refinement of the above to fields of positive characteristic, $p$. Using eigenvalues of Jucys-Murphy elements, the induction and restriction functors decompose into $p$ summands. In this way, the above gives rise to a specific representation of the affine Kac-Moody algebra $\hat{\mathfrak{sl}}_p$,  (cf. \cite{LLT95}, \cite{LLT96}, \cite{Ar96}, and \cite{Gr99}). This approach has since spawned many interesting results connecting the representation theories of various Hecke algebras with those of other algebras (cf., e.g. \cite{Ar01}, \cite{Ja05}, \cite{Ar06}, \cite{BK09}, \cite{Kh14}, \cite{RS15}, and \cite{MV18}).

The original motivation for the present paper comes from the attempt to investigate a similar construction (in the case of modules over complex numbers) for dihedral groups rather than the symmetric groups. There are several significant differences. While the symmetric groups are naturally included into each other with respect to the usual linear order on their set of indices (the set of nonnegative integers), natural inclusions of dihedral groups are given by the divisibility partial order on ther set of their orders. As a consequence, we have infinitely many ``elementary'' induction and restriction functors, naturally indexed by prime numbers. The aim of this paper is to understand the basic combinatorics which these functors generate on the level of the Grothendieck group. 

\subsection{Contents}
The paper is structured as follows. In Section \ref{s2}, we fix some notation and recall basic facts about the dihedral groups, their modules, and restrictions and inductions of the latter. In Section \ref{s3}, we describe the actions of the restriction and induction functors on all simple modules over the dihedral groups. We use this to define the algebras $A_{P,\mathcal{M}}$, which depend on a choice of a set $P$ of prime numbers and a collection (satisfying a few closedness conditions) $\mathcal{M}$ of simple dihedral group modules.
These algebras will be the main objects of study in this paper. 

Section \ref{s4} contains a couple of results on these algebras which can be obtained without restrictions on the orders of the involved dihedral groups. Section \ref{s5} is devoted to the study of the case of dihedral groups of the form $D_{2n}$ where $n$ is odd. We give a presentation and describe a basis of $A_{P,\mathcal{M}}$ in Theorems \ref{babybasisthm} and \ref{basisthm}. Furthermore, we describe the center of $A_{P,\mathcal{M}}$ in Theorem \ref{cthm} and use central idempotents to obtain a decomposition of $A_{P,\mathcal{M}}$ into a direct sum of two indecomposable algebras in Theorems \ref{babydecompthm} and \ref{decompthm}. In Corollary \ref{bicyccor}, we see that in certain cases, the indecomposable components of $A_{P,\mathcal{M}}$ can be described as tensor powers of the semigroup algebra of the classical bicyclic monoid (cf., e.g, \cite{CP64}). Finally, in Section \ref{s6}, we discuss the more difficult case involving all dihedral groups, tie up some loose ends, and speculate on possible further directions of study.

\subsection*{Acknowledgements}
The author is very much indebted to his advisor, Volodymyr Mazorchuk, for valuable discussions on the content of the paper as well as its presentation.

\section{Preliminaries}\label{s2}
\subsection{Miscellaneous notation, assumptions, and conventions}
By $\N$ we denote the set of nonnegative integers; the set of positive integers we denote by $\Z_{>0}$. 

We use double brackets to denote intervals (open, closed or half-open) of integers. For instance $\lb 1,4\rp=\{1,2,3\}$. 

All vector spaces (in particular modules, algebras etc) considered will be complex. 

All modules considered will be left modules.

By angled brackets $\langle A|B\rangle$ we mean the algebraic structure (of a kind specified by the context) generated by the elements $A$ subject to the relations $B$.

\subsection{Dihedral groups}
For each integer $n\ge 3$, the dihedral group $D_{2n}$ is defined by 
\begin{equation*}
D_{2n}=\langle r_n,s_n|r_n^n=1, s_nr_ns_n=r_n^{-1}\rangle,
\end{equation*}
and may be identified with its natural (real) representation, which is the group of symmetries of the regular $n$-gon inscribed in the unit circle such that $(1,0)$ is a vertex. Under this identification, $r_n$ corresponds to a rotation by $2\pi/n$, and $s_n$ corresponds to reflection with respect to the horizontal axis. Note that we consider dihedral groups for $n=1,2$ undefined, in contrast to going the Coxeter route where it is natural to define dihedral groups also for these $n$. This will have important consequences for the structure of our main objects of study. For a brief discussion on the case of defined dihedral groups for $n=1,2$, see Section \ref{s6}.

\subsection{Modules over dihedral groups}
For any integer $n\ge 3$, let us define $V_{a,b}(n)$ to be the one-dimensional complex $D_{2n}$-module with $r_n$-action given by multiplication with $a$ and $s_n$-action given by multiplication with $b$. Here $b\in\{1,-1\}$, and $a=1$ if $n$ is odd while $a\in\{1,-1\}$ if $n$ is even. 

Also define for any integers $k$ and $n\ge 3$ the two-dimensional complex $D_{2n}$-module $W_k(n)$ with $r_n$-action given by $\begin{pmatrix} e^{2\pi i k/n}&0\\0&e^{-2\pi i k/n}\end{pmatrix}$ and $s_n$-action given by $\begin{pmatrix} 0&1\\1&0\end{pmatrix}$ (both matrices are with respect to the standard basis). 

Let us for technical reasons also define $V_{a,b}(n)=0$ and $W_k(n)=0$ whenever $n$ is not an integer greater than or equal to 3, or (in the latter case) when $k$ is not an integer.

The modules $W_k(n)$ are further described in the following easy Lemma. 
\begin{mylem}
\label{wlem}
\begin{enumerate}
\item[$($i$)$]
If $k\equiv \pm l\text{ (mod $n$)}$, then $W_k(n)\cong W_l(n)$.
\item[$($ii$)$]
The module $W_k(n)$ is indecomposable (hence simple) if $k\not\in\frac{1}{2}\Z n$. 
\item[$($iii$)$]
If $k\in \Z n$, then $W_k(n)\cong V_{1,1}(n)\oplus V_{1,-1}(n)$. 
\item[$($iv$)$]
If $k\in \frac{1}{2}\Z n\backslash \Z n$, then $W_k(n)\cong V_{-1,1}(n)\oplus V_{-1,-1}(n)$.
\end{enumerate}
\end{mylem}
\begin{proof}
Statement (i) holds because if $k\equiv l\text{ (mod $n$)}$ then $\id:W_k(n)\xrightarrow{\sim} W_l(n)$, and if $k\equiv -l\text{ (mod $n$)}$ then $s_n\cdot:W_k(n)\xrightarrow{\sim} W_l(n)$. 

Statements (ii), (iii) and (iv) hold because $\begin{pmatrix} 0&1\\1&0\end{pmatrix}$ has eigenvectors $(1,1)$ and $(1,-1)$, neither of which is an eigenvector of $\begin{pmatrix} e^{2\pi i k/n}&0\\0&e^{-2\pi i k/n}\end{pmatrix}$ unless $k\in\frac{1}{2}\Z n$, in which case they form a basis of $V_{1,1}(n)$ and $V_{1,-1}(n)$ respectively $V_{-1,1}(n)$ and $V_{-1,-1}(n)$. 
\end{proof}

The classification of simple $D_{2n}$-modules is given in the following proposition, and a proof can be found e.g. as Theorems 3.4.1 and 3.4.2 in \cite{So14}. 
\begin{myprop}
The simple $D_{2n}$-modules have either dimension 1 or 2. These are of the forms
\begin{enumerate}
\item[$($i$)$]
$V_{a,b}(n)$, for any integer $n\ge 3$, for $b\in\{1,-1\}$, for $a=1$ in case of odd $n$, and for $a\in\{1,-1\}$ in case of even $n$.
\item[$($ii$)$]
$W_k(n)$, for any integer $n\ge 3$ and $k\in \lb 1,n/2\rp$.  
\end{enumerate}
Also, these modules are nonisomorphic. 
\end{myprop}

\section{The restriction and induction functors}\label{s3}
Let $\dmod$ denote the category of all left $D_{2n}$-modules. For any integer $p$, there is a natural inclusion
\begin{equation}
\label{incleq}
\begin{aligned}
D_{2n}&\hookrightarrow D_{2pn}\\
r_n&\mapsto r_{pn}^{p}\\
s_n&\mapsto s_{pn}.
\end{aligned}
\end{equation}
Since any such inclusion factors into ones where $p$ is prime, we will throughout the rest of this text without loss of generality assume that $p$ is prime. 

With respect to these inclusions, we have the induction and restriction functors 
\begin{align*}
\ind_n^{pn}: D_{2n}\mmod&\rightarrow D_{2pn}\mmod\\
M&\mapsto D_{2pn}\otimes_{\C[D_{2n}]} M
\end{align*}
and
\begin{align*}
\res_n^{pn}: D_{2pn}\mmod&\rightarrow D_{2n}\mmod\\
M&\mapsto M_{|D_{2n}}=D_{2n}\otimes_{\C[D_{2n}]} M
\end{align*}
respectively. In particular, it is understood that we will only consider induction and restriction between dihedral groups whose order differ by a prime factor. In what follows we will -- somewhat sloppily -- write $\res_p$ and $\ind_p$ instead of $\res_n^{pn}$ and $\ind_n^{pn}$ whenever the intended functors should be clear from the context.

The following proposition describes the actions of the induction and restriction functors on simple modules.
\begin{myprop}
\label{resindprop}
Restriction and induction act as follows on simple dihedral group modules. 
\begin{enumerate}
\item[$($i$)$]
$\res_p V_{a,b}(np)\cong\begin{cases}
        V_{a,b}(n), & \mbox{if } p\ne 2,\\
        V_{|a|,b}(n) & \mbox{if } p=2.
        \end{cases}$
\item[$($ii$)$]
$\res_p W_k(np)\cong W_k(n)\cong \begin{cases}
        W_{\pm k \text{ (mod $n$)}\in \lb 1,n/2\rp}(n), & \mbox{if } k\not\in\frac{1}{2}\Z n,\\
        V_{1,1}(n)\oplus V_{1,-1}(n) & \mbox{if } k\in\Z n,\\
        V_{-1,1}(n)\oplus V_{-1,-1}(n) & \mbox{if $n$ is even and } k\in\frac{1}{2}\Z n\backslash \Z n.
        \end{cases}$
\item[$($iii$)$]
$\ind_p V_{a,b}(n)\cong \begin{cases}
        V_{a,b}(pn)\oplus\bigoplus_{j\in \lb 1,\frac{p-1}{2}\rb}W_{jn}(pn), & \mbox{if } p\ne 2\text{ and }a=1,\\
        V_{a,b}(pn)\oplus\bigoplus_{j\in \lb 1,\frac{p-1}{2}\rb}W_{(j-\frac{1}{2})n}(pn), & \mbox{if }p\ne 2\text{ and }a=-1\text{ (where $n$ is even)},\\
        V_{1,b}(pn)\oplus V_{-1,b}(pn) & \mbox{if } p=2\text{ and }a=1,\\
        W_{n/2}(pn) & \mbox{if } p=2\text{ and }a=-1\text{ (where $n$ is even)}.
        \end{cases}$
\item[$($iv$)$]
$\ind_p W_k(n)\cong \begin{cases}
        W_k(pn)\oplus\bigoplus_{j\in\lb 1,\frac{p-1}{2}\rb}(W_{-k+jn}(pn)\oplus W_{k+jn}(pn)), & \mbox{if } p\ne 2,\\
        W_k(pn)\oplus W_{-k+n}(pn) & \mbox{if } p=2.
        \end{cases}$
\end{enumerate}
\end{myprop}
\begin{proof}
{\bf Part (i):} The generator $r_n$ acts on $V_{a,b}(pn)_{|D_{2n}}$ via $r_{pn}^p$, which acts by $a$ if $p$ is odd, and by $|a|$ if $p$ is even. The generator $s_n$ acts on $V_{a,b}(pn)_{|D_{2n}}$ via $s_{pn}$, which acts by $b$. 

{\bf Part (ii):} The generator $r_n$ acts on $W_k(pn)_{|D_{2n}}$ via $r_{pn}^p$, which acts by $\begin{pmatrix} e^{2\pi i kp/(pn)}&0\\0&e^{-2\pi i kp/(pn)}\end{pmatrix}=\begin{pmatrix} e^{2\pi i k/n}&0\\0&e^{-2\pi i k/n}\end{pmatrix}$. The generator $s_n$ acts on $W_k(pn)_{|D_{2n}}$ via $s_{pn}$, which acts by $\begin{pmatrix} 0&1\\1&0\end{pmatrix}$. Hence we have the isomorphism $\res_p W_k(np)\cong W_k(n)$. The rest now follows from Lemma \ref{wlem}.

{\bf Part (iii):} This is done by applying Frobenius reciprocity.

{\bf Part (iv):} Likewise. 
\end{proof}
The action of various inductions and restrictions on the simple modules over the dihedral groups define a partial order on those simple modules: for modules $M$ and $N$ we define $M\le N$ if and only if $M\cong N$ or $M$ is a summand of some restriction of $N$. We may conveniently illustrate these actions using the Hasse diagrams with respect to these partial orders; this is done in Figures 1 and 2. These diagrams are analogous to the Bratteli diagrams used for instance to study restriction and induction in the case of symmetric groups (see for instance \cite{BS05}), but here the underlying ordering of the group algebras is not linear and furthermore the trivial group algebra is not included. We will call these graphs \emph{induction/restriction diagrams}.  
\begin{sidewaysfigure}
\label{figure1}
\begin{centering}

\xymatrix@!=1pc@C=0.01pt{
&&&&&&&&&&&&&&&&&&&&&&&&&&&&&&&&&&&&&&&&&&&&&&&&&&&&&&&&&&&\\
&&&&&&&&&&&&&&&&&&&&&&&&&&&&&&&&&&&&&&&&&&&&&&&&&&&&&&&&&&&\\
&&&&&&&&&&&&&&&&&&&&&&&&&&&&&&&&&&&&&&&&&&&&&&&&&&&&&&&&&&&\\
&&&&&&&&&&&&&&&&&&&&&&&&&&&&&&&&&&&&&&&&&&&&&&&&&&&&&&&&&&&\\
&&&&&&&&&&&&&&&&&&&&&&&&&&&&&&&&&&&&&&&&&&&&&&&&&&&&&&&&&&&\\
&&&&&&&&&&&&&&&&&&&&&&&&&&&&&&&&&&&&&&&&&&&&&&&&&&&&&&&&&&&\\
&&&&&&&&&&&&&&&&&&&&&&&&&&&&&&&&&&&&&&&&&&&&&&&&&&&&&&&&&&&\\
&&&&&&&&&&&&&&&&&&&&&&&&&&&&&&&&&&&&&&&&&&&&&&&&&&&&&&&&&&&\\
&&&&&&&&&&&&&&&&&&&&&&&&&&&&&&&&&&&&&&&&&&&&&&&&&&&&&&&&&&&\\
&&&&&&&&&&&&&&&&&&&&&&&&&&&&&&&&&&&&&&&&&&&&&&&&&&&&&&&&&&&\\
&&&&&&&&&&&&&&&&&&&&&&&&&&&&&&&&&&&&&&&&&&&&&&&&&&&&&&&&&&&\\
&&&&&&&&&&&&&&&&&&&&&&&&&&&&&&&&&&&&&&&&&&&&&&&&&&&&&&&&&&&\\
&&&&&&&&&&&&&&&&&&&&&&&&&&&&&&&&&&&&&&&&&&&&&&&&&&&&&&&&&&&\\
&&{\scriptstyle W_{1}(pm)}\ar@{-}[ul]\ar@{-}[ull]\ar@{-}[u]\ar@{-}[ur]\ar@{-}[urr]&&&&&{\scriptstyle W_{-1+m}(pm)}\ar@{-}[ul]\ar@{-}[ull]\ar@{-}[u]\ar@{-}[ur]\ar@{-}[urr]&&&&&{\scriptstyle W_{1+m}(pm)}\ar@{-}[ul]\ar@{-}[ull]\ar@{-}[u]\ar@{-}[ur]\ar@{-}[urr]&&&&&{\scriptstyle W_{-1+2m}(pm)}\ar@{-}[ul]\ar@{-}[ull]\ar@{-}[u]\ar@{-}[ur]\ar@{-}[urr]&&&&&{\scriptstyle \dots}&&&&&{\scriptstyle W_{1+\frac{p-1}{2}m}(pm)}\ar@{-}[ul]\ar@{-}[ull]\ar@{-}[u]\ar@{-}[ur]\ar@{-}[urr]&&&&&&&&&&&&&&&&&&&&&&&&&&&&&&&&\\
&&&&&&&&&&&&&&&&&{\scriptstyle \dots}&&&&&&&&&&&&&&&&&&&&&&&&&&&&&&&&&&&&&&&&&&\\
&&&&&&&&&&&&&&{\scriptstyle W_1(m)}\ar@{-}[uullllllllllll]\ar@{-}[uulllllll]\ar@{-}[uull]\ar@{-}[uurrr]\ar@{-}[uurrrrrrrrrrrrr]&&&&&&&&&&&&&&&&&&&&&&&{\scriptstyle W_2(m)}\ar@{-}[ul]\ar@{-}[ull]\ar@{-}[u]\ar@{-}[ur]\ar@{-}[urr]&&&&\dots&&&&{\scriptstyle W_{\frac{m-1}{2}}(m)}\ar@{-}[ul]\ar@{-}[ull]\ar@{-}[u]\ar@{-}[ur]\ar@{-}[urr]&&&&&&&&&&&&&&\\
&&&&&&&&&&&&&&&&&&&&&&&&&&&&&&&&&&&&&&&&&&&&&&&&&&&&&&&&&&&\\
&&{\scriptstyle W_m(p^2m)}\ar@{-}[ul]\ar@{-}[ull]\ar@{-}[u]\ar@{-}[ur]\ar@{-}[urr]&&&&&{\scriptstyle W_{-m+pm}(p^2m)}\ar@{-}[ul]\ar@{-}[ull]\ar@{-}[u]\ar@{-}[ur]\ar@{-}[urr]&&&&&{\scriptstyle W_{m+pm}(p^2m)}\ar@{-}[ul]\ar@{-}[ull]\ar@{-}[u]\ar@{-}[ur]\ar@{-}[urr]&&&&&{\scriptstyle W_{-m+2pm}(p^2m)}\ar@{-}[ul]\ar@{-}[ull]\ar@{-}[u]\ar@{-}[ur]\ar@{-}[urr]&&&&&{\scriptstyle \dots}&&&&&{\scriptstyle W_{m+\frac{p-1}{2}pm}(p^2m)}\ar@{-}[ul]\ar@{-}[ull]\ar@{-}[u]\ar@{-}[ur]\ar@{-}[urr]&&&&&{\scriptstyle V_{1,1}(p^2m)}\ar@{-}[ul]\ar@{-}[ull]\ar@{-}[u]\ar@{-}[ur]\ar@{-}[urr]&&&&&{\scriptstyle W_{pm}(p^2m)}\ar@{-}[ul]\ar@{-}[ull]\ar@{-}[u]\ar@{-}[ur]\ar@{-}[urr]&&&&&{\scriptstyle W_{2pm}(p^2m)}\ar@{-}[ul]\ar@{-}[ull]\ar@{-}[u]\ar@{-}[ur]\ar@{-}[urr]&&&&&{\scriptstyle \dots}&&&&&{\scriptstyle W_{\frac{p-1}{2}pm}(p^2m)}\ar@{-}[ul]\ar@{-}[ull]\ar@{-}[u]\ar@{-}[ur]\ar@{-}[urr]&&&&&{\scriptstyle V_{1,-1}(p^2m)}\ar@{-}[ul]\ar@{-}[ull]\ar@{-}[u]\ar@{-}[ur]\ar@{-}[urr]&&\\
&&&&&&&&&&&&&&&&&&&&&&&&&&&&&&&{\scriptstyle \dots}&&&&&&&&{\scriptstyle \dots}&&&&&&&&&&&&&&&&&&&&\\
&&&&&&&&&&&&&&&&&&&&&&&&&&&&&&&&{\scriptstyle V_{1,1}(pm)}\ar@{-}[uu]\ar@{-}[uurrrrr]\ar@{-}[uurrrrrrrrrr]\ar@{-}[uurrrrrrrrrrrrrrrrrrrr]&&&&&{\scriptstyle W_m(pm)}\ar@{-}[uullllllllllllllllllllllllllllllllll]\ar@{-}[uulllllllllllllllllllllllllllll]\ar@{-}[uullllllllllllllllllllllll]\ar@{-}[uulllllllllllllllllll]\ar@{-}[uulllllllll]&&&&&{\scriptstyle W_{2m}(pm)}\ar@{-}[ul]\ar@{-}[ull]\ar@{-}[u]\ar@{-}[ur]\ar@{-}[urr]&&&&&{\scriptstyle \dots}&&&&&{\scriptstyle W_{\frac{p-1}{2}}(pm)}\ar@{-}[ul]\ar@{-}[ull]\ar@{-}[u]\ar@{-}[ur]\ar@{-}[urr]&&&&&{\scriptstyle V_{1,-1}(pm)}\ar@{-}[uu]\ar@{-}[uulllll]\ar@{-}[uulllllllllllllll]\ar@{-}[uullllllllllllllllllll]&&\\
&&&&&&&&&&&&&&&&&&&&&&&&&&&&&&&&&&&&&&&{\scriptstyle \dots}&&&&&&&&&&&&&&{\scriptstyle \dots}&&&&&&\\
&&&&&&&&&&&&&&&&&&&&&&&&&&&&&&&&{\scriptstyle V_{1,1}(m)}\ar@{-}[uu]\ar@{-}[uurrrrr]\ar@{-}[uurrrrrrrrrr]\ar@{-}[uurrrrrrrrrrrrrrrrrrrr]&&&&&&&&&&&&&&&&&&&&&&&&&{\scriptstyle V_{1,-1}(m)}\ar@{-}[uu]\ar@{-}[uulllll]\ar@{-}[uulllllllllllllll]\ar@{-}[uullllllllllllllllllll]&&
}
\end{centering}
\caption{The induction/restriction diagram of all simple $D_{2mp^l}$-modules, where $p$ is an odd prime, where $l\in\N$, and where $m$ is either equal to $p$ or odd and not divisible by $p$. The diagram has $\frac{m+1}{2}$ connected components.}
\end{sidewaysfigure}
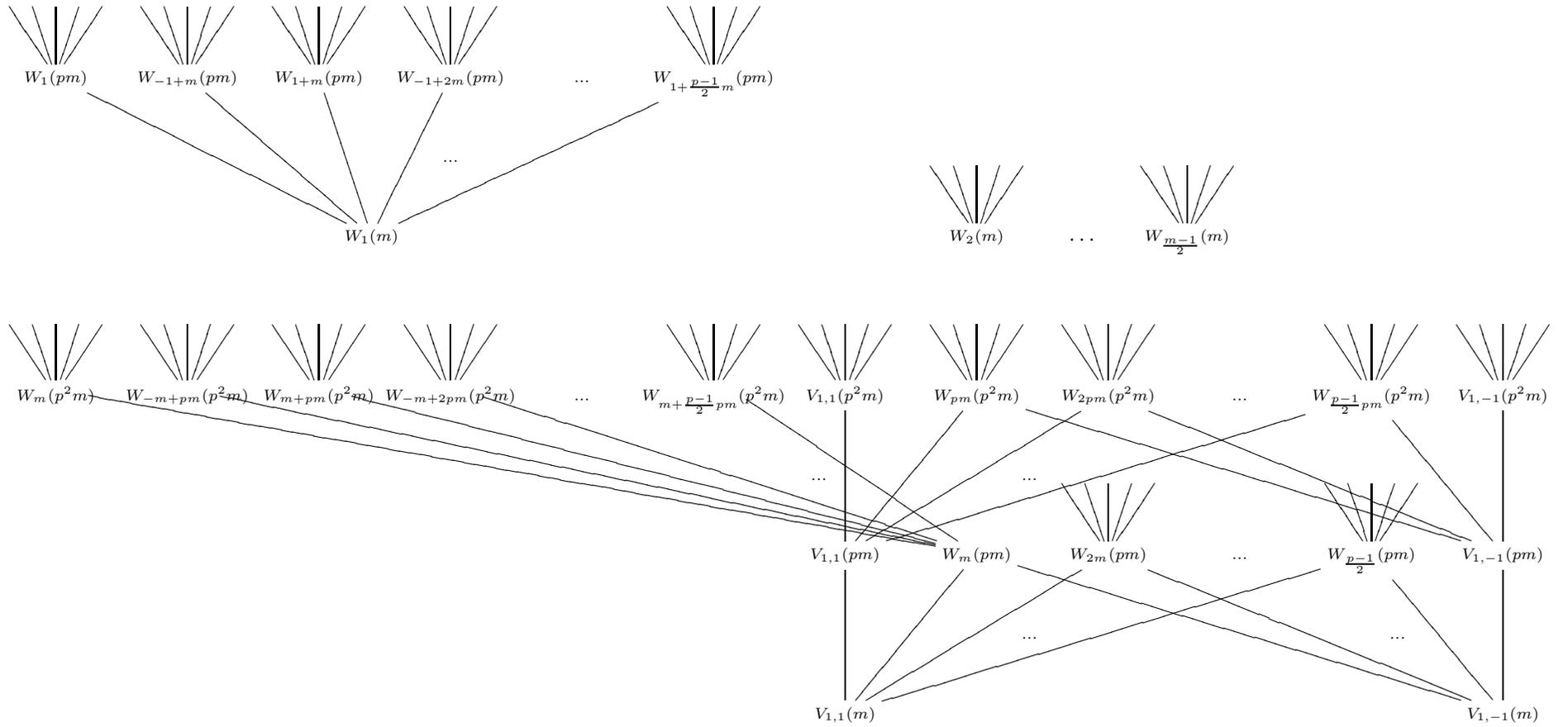

\begin{sidewaysfigure}
\label{figure2}
\begin{centering}

\xymatrix@!=1pc@C=0.01pt{
&&&&&&&&&&&&&&&&&&&&&&&&&&&&&&&&&&&&&&&&&&&&&&&&&&&&&&&&&&&&&&&\\
&&&&&&&&&&&&&&&&&&&&&&&&&&&&&&&&&&&&&&&&&&&&&&&&&&&&&&&&&&&&&&&\\
&&&&&&&&&&&&&&&&&&&&&&&&&&&&&&&&&&&&&&&&&&&&&&&&&&&&&&&&&&&&&&&\\
&&&&&&&&&&&&&&&&&&&&&&&&&&&&&&&&&&&&&&&&&&&&&&&&&&&&&&&&&&&&&&&\\
&&&&&&&&&&&&&&&&&&&&&&&&&&&&&&&&&&&&&&&&&&&&&&&&&&&&&&&&&&&&&&&\\
&&&&&&&&&&&&&&&&&&&&&&&&&&&&&&&&&&&&&&&&&&&&&&&&&&&&&&&&&&&&&&&\\
&&&&&&&&&&&&&&&&&&&&&&&&&&&&&&&&&&&&&&&&&&&&&&&&&&&&&&&&&&&&&&&\\
&&&&&&&&&&&&&&&&&&&&&&&&&&&&&&&&&&&&&&&&&&&&&&&&&&&&&&&&&&&&&&&\\
&&&&&&&&&&&&&&&&&&&&&&&&&&&&&&&&&&&&&&&&&&&&&&&&&&&&&&&&&&&&&&&\\
&&&&&&&&&&&&&&&&&&&&&&&&&&&&&&&&&&&&&&&&&&&&&&&&&&&&&&&&&&&&&&&\\
&&&&&&&&&&&&&&&&&&&&&&&&&&&&&&&&&&&&&&&&&&&&&&&&&&&&&&&&&&&&&&&\\
&&&&&&&&&&&&&&&&&&&&&&&&&&&&&&&&&&&&&&&&&&&&&&&&&&&&&&&&&&&&&&&\\
&&&&&&&&&&&&&&&&&&&&&&&&&&&&&&&&&&&&&&&&&&&&&&&&&&&&&&&&&&&&&&&\\
{\scriptstyle V_{1,1}(8m)}\ar@{-}[ur]\ar@{-}[u]&&&&{\scriptstyle V_{-1,1}(8m)}\ar@{-}[ur]&&&&{\scriptstyle W_{m}(8m)}\ar@{-}[ul]\ar@{-}[ur]&&&&{\scriptstyle W_{2m}(8m)}\ar@{-}[ul]\ar@{-}[ur]&&&&{\scriptstyle W_{-m+4m}(8m)}\ar@{-}[ul]\ar@{-}[ur]&&&&{\scriptstyle V_{-1,-1}(8m)}\ar@{-}[ul]&&&&{\scriptstyle V_{1,-1}(8m)}\ar@{-}[ul]\ar@{-}[u]&&&&&&&&&&&&&&&&&&&&&&&&&&&&&&&&&&&&&&&\\
&&&&&&&&&&&&&&&&&&&&&&&&&&&&&&&&&&&&&&&&&&&&&&&&&&&&&&&&&&&&&&&\\
{\scriptstyle V_{1,1}(4m)}\ar@{-}[uu]\ar@{-}[uurrrr]&&&&{\scriptstyle V_{-1,1}(4m)}\ar@{-}[uurrrrrrrr]&&&&&&&&{\scriptstyle W_{m}(4m)}\ar@{-}[uullll]\ar@{-}[uurrrr]&&&&&&&&{\scriptstyle V_{-1,-1}(4m)}\ar@{-}[uullllllll]&&&&{\scriptstyle V_{1,-1}(4m)}\ar@{-}[uu]\ar@{-}[uullll]&&&&&{\scriptstyle W_1(4m)}\ar@{-}[ul]\ar@{-}[ur]&&&&{\scriptstyle W_{-1+2m}(4m)}\ar@{-}[ul]\ar@{-}[ur]&&&&{\scriptstyle W_{-1+m}(4m)}\ar@{-}[ul]\ar@{-}[ur]&&&&{\scriptstyle W_{1+m}(4m)}\ar@{-}[ul]\ar@{-}[ur]&&&&&&&&\textcolor{gray}{{\scriptstyle W_{-1+4}(8)}}\ar@{-}@[gray][ur]\ar@{-}@[gray][ul]&&&&\textcolor{gray}{{\scriptstyle W_{1+4}(8)}}\ar@{-}@[gray][ur]\ar@{-}@[gray][ul]&&&&&&&&&&\\
&&&&&&&&&&&&&&&&&&&&&&&&&&&&&&&&&&&&&&&&&&&&&&&&&&&&&&&&&&&&&&&\\
{\scriptstyle V_{1,1}(2m)}\ar@{-}[uu]\ar@{-}[uurrrr]&&&&{\scriptstyle V_{-1,1}(2m)}\ar@{-}[uurrrrrrrr]&&&&&&&&&&&&&&&&{\scriptstyle V_{-1,-1}(2m)}\ar@{-}[uullllllll]&&&&{\scriptstyle V_{1,-1}(2m)}\ar@{-}[uu]\ar@{-}[uullll]&&&&&&&{\scriptstyle W_1(2m)}\ar@{-}[uull]\ar@{-}[uurr]&&&&&&&&{\scriptstyle W_{-1+m}(2m)}\ar@{-}[uull]\ar@{-}[uurr]&&&&&&&&&&&&\textcolor{gray}{{\scriptstyle W_1(4)}}\ar@{-}@[gray][uurr]\ar@{-}@[gray][uull]&&&&&&&&&&&&\\
&&&&&&&&&&&&&&&&&&&&&&&&&&&&&&&&&&&&&&&&&&&&&&&&&&&&&&&&&&&&&&&\\
{\scriptstyle V_{1,1}(m)}\ar@{-}[uu]\ar@{-}[uurrrr]&&&&&&&&&&&&&&&&&&&&&&&&{\scriptstyle V_{1,-1}(m)}\ar@{-}[uu]\ar@{-}[uullll]&&&&&&&&&&&{\scriptstyle W_1(m)}\ar@{-}[uullll]\ar@{-}[uurrrr]&&&& \dots&&&&{\scriptstyle W_{\frac{m-1}{2}}(m)}\ar@{-}[ul]\ar@{-}[ur]&&&&&&\textcolor{gray}{{\scriptstyle V_{-1,1}(2)}}\ar@{-}@[gray][uurr]&&&&\textcolor{gray}{{\scriptstyle V_{-1,-1}(2)}}\ar@{-}@[gray][uull]&&&&&&&&&&
}
\end{centering}
\caption{The induction/restriction diagram of all simple $D_{2mp^l}$-modules, where $p=2$, where $l\in\N$, and where $m$ is either equal to $p$ or not divisible by $p$. The diagram has $\frac{m+1}{2}$ connected components; the grayed out (rightmost) component is excluded in case $m\ne p$, and included in place of the $\frac{m-1}{2}$ components to its left in case $m=p$.} 
\end{sidewaysfigure}
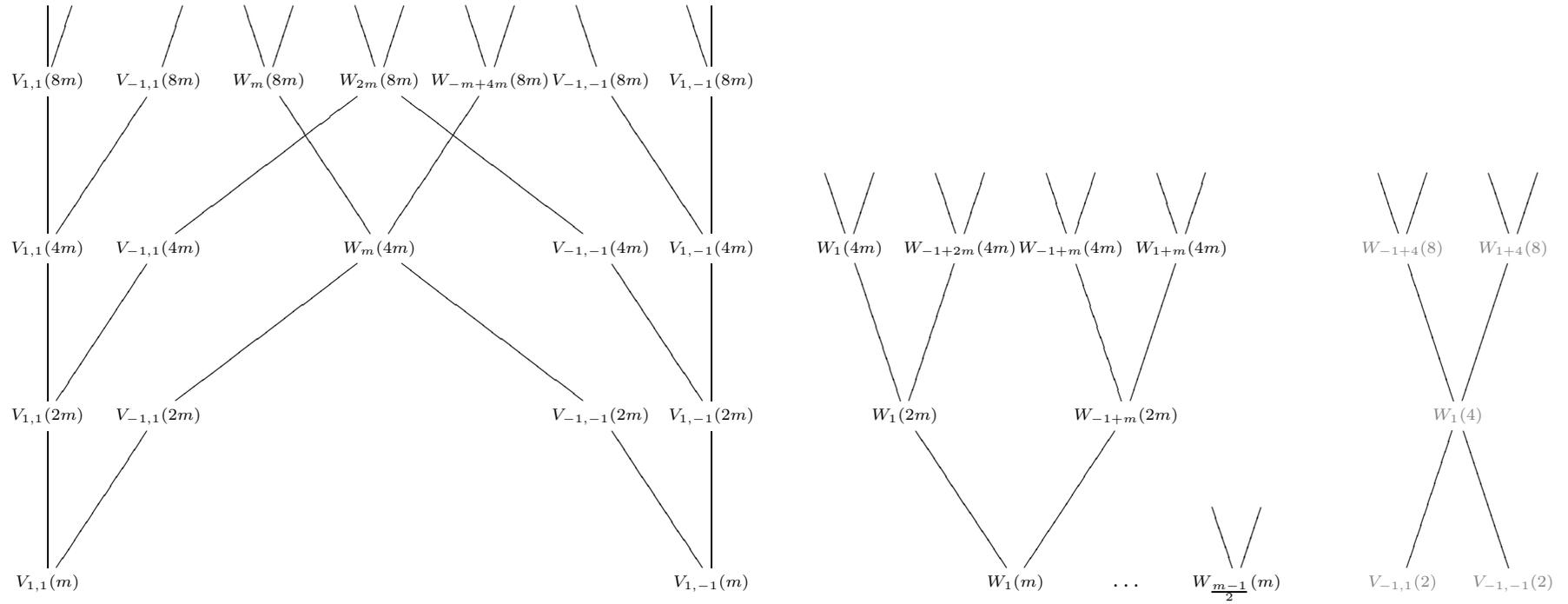

\subsection{Induction and restriction module structure on Grothendieck groups of dihedral group modules}
Define induction and restriction functors $\ind_p$ and $\res_p$ on $\bigoplus_{n\ge 3} D_{2n}\mmod$ by setting
\begin{equation*}
{\ind_p}_{|D_{2n}\mmod}=\ind_n^{pn}
\end{equation*}
and
\begin{equation*}
{\res_p}_{|D_{2n}\mmod}=\begin{cases}
       \res_{n/p}^n, & \mbox{if }p|n\text{ and }n\ne p,\\
        0, & \mbox{otherwise}.
        \end{cases},
\end{equation*}
and extending via additivity. These functors then also induce endomorphisms on the Grothendieck group 
\begin{equation*}
\groth[\bigoplus_{n\ge 3} D_{2n}\mmod],
\end{equation*}
where we note that the split Grothendieck group and the regular Grothendieck group coincide because of Maschke's Theorem. By further abuse of notation, the induced functors will also be denoted by $\res_p$ and $\ind_p$ respectively. 

\subsection{The algebras \texorpdfstring{$A_{P,\mathcal{M}}$}{APM}}
For $P$ being any set of primes, define $A_P$ to be the free algebra generated by the symbols $\res_p$ and $\ind_p$ with $p\in P$. By abuse of notation, let us sometimes omit the set brackets of singletons and also write $A_p=A_{\{p\}}$. 

Then the complexified Grothendieck group
\begin{equation*}
\mathcal{G}=\C\otimes_{\Z}\groth[\bigoplus_{n\ge 3} D_{2n}\mmod]
\end{equation*}
becomes an $A_P$-module with action induced by the actions of $\res_p$ and $\ind_p$ on the Grothendieck group. For any submodule $\mathcal{M}\subset \mathcal{G}$, let $\ann_{A_P}(\mathcal{M})$ be the ideal of elements of $A_P$ that annihilate each element of $\mathcal{M}$, and let
\begin{equation*}
A_{P,\mathcal{M}}=A_{P}/\ann_{A_P}(\mathcal{M}).
\end{equation*} 
The action of $A_P$ on $\mathcal{M}$ induces an action of $A_{P,\mathcal{M}}$ on $\mathcal{M}$ in the obvious way.

In what follows we will study the algebra $A_{P,\mathcal{M}}$ as well as $\ann_{A_P}(\mathcal{M})$, the latter being the kernel of the natural surjection
\begin{equation*}
\varphi_{P,\mathcal{M}}:A_P\rightarrow A_{p,\mathcal{M}},
\end{equation*}
We observe that for any $z_1,z_2\in A_P$ we have that $\varphi_{P,\mathcal{M}}(z_1)=\varphi_{P,\mathcal{M}}(z_2)$ if and only if $z_1M=z_2M$ for all $M\in\mathcal{M}$, a fact that will be used extensively in the proofs to follow. 

By further abuse of notation, we will often let $\res_p$ and $\ind_p$ denote also the images $\varphi_{P,\mathcal{M}}(\res_p)$ and $\varphi_{P,\mathcal{M}}(\ind_p)$ respectively when no confusion should occur. 

The following lemma is obvious.
\begin{mylem}
\label{wloglem}
Let $P$ be some set of primes, let $\mathcal{N}\subset\mathcal{M}\subset\mathcal{G}$ be $A_P$-submodules and let $z\in\ker(\varphi_{P,\mathcal{M}})$. Then also $z\in\ker(\varphi_{P,\mathcal{N}})$.
\end{mylem}

It is clear that the representation of $A_P$ by $\mathcal{M}$ factors through $A_{P,\mathcal{M}}$ via a natural surjection and that $A_{P,\mathcal{M}}$ is terminal with this property. The following proposition offers another way to think about $A_{P,\mathcal{M}}$. 
\begin{myprop}
The $A_P$-module $A_{P,\mathcal{M}}$ satisfies that for any nonzero $a\in A_{P,\mathcal{M}}$ there exists a homomorphism of $A_P$-modules
\begin{equation*}
A_{P,\mathcal{M}}\rightarrow \mathcal{M}
\end{equation*}
which does not annihilate $a$, and $A_{P,\mathcal{M}}$ is the unique maximal quotient of the regular $A_{P}$-module that satisfies this property.  
\end{myprop}
\begin{proof}
The homomorphisms $A_{P,\mathcal{M}}\rightarrow \mathcal{M}$ correspond precisely to the mappings $1\mapsto M\in\mathcal{M}$ by the extension $a\mapsto aM$ for all $a\in A_{P,\mathcal{M}}$. The quotient $A_{P,\mathcal{M}}$ of $A_P$ is taken by precisely the set of $a\in A_P$ for which $aM=0$ for every $M\in\mathcal{M}$, and these $a$ are those for which there can not exist a homomorphism of $A_P$-modules $A_{P,\mathcal{M}}\rightarrow \mathcal{M}$
which does not annihilate $a$.
\end{proof}

\subsection{Termini and nadirs}
We conclude the present section by introducing certain features of monomials in $A_P$, which will be very useful to consider in coming proofs. Throughout this subsection, let $z\in A_P$ be a (non-commutative) nonzero monomial, i.e. up to a scalar a sequence of various $\res_q$ and $\ind_q$ symbols, where $q$ ranges over $P$, and let $p\in P$ be fixed. 

We call the total number of $\ind_p$ in $z$ minus the total number of $\res_p$ in $z$ the \emph{terminus} of $z$ with respect to $p$. Such a terminus will most often be denoted by $e_p$.

We call a terminal subsequence (i.e. a right monomial factor) $z'$ of $z$ a \emph{nadir in} $z$ with respect to $p$ if the number of $\ind_p$ minus the number of $\res_p$ in $z'$ is minimal over all terminal subsequences of $z$. This number will be called the \emph{nadir of} $z$ with respect to $p$, and will most often be denoted by $d_p$. Note that the word nadir will thus be used in two different (albeit related) ways distinguished by the choice of preposition. In particular, the nadir of a sequence with respect to a fixed prime number is unique, while a nadir in the subsequence with respect to that prime is not necessarily unique. If $z'$ is a nadir in $z$ with respect to all $p\in P$ simultaneously, then we call $z'$ a \emph{total nadir} in $z$.

The following lemma is a first example of results which rely on these concepts.

\begin{mylem}
\label{nadirlem}
Let $P$ be any set of primes, and let $z\in A_P$ be a monomial. For $p\in P$, let $d_p$ be the nadir of $z$ with respect to $p$.
\begin{enumerate}
\item[$($i$)$]
If there is no total nadir in $z$, then, for every simple $D_{2n}$-module $L$, we have $zL=0$ if and only if $p^{-d_p}\not|n$ for some $p\in P$.
\item[$($ii$)$]
If there is a a total nadir in $z$, then, for any simple $D_{2n}$-module $L$, we have that $zL=0$ if and only if $p^{-d_p}\not|n$ for some $p\in P$ or $n=2\prod_{p\in P}p^{-d_p}$. 
\end{enumerate}
\end{mylem}
\begin{proof}
For $z$ of degree 1, the statement of the Lemma is clear from the definitions of the actions of $\res_p$ and $\ind_p$ on simple $D_{2n}$-modules. From Proposition \ref{resindprop} we see that the structure constants of these actions are all nonnegative, so $zL=0$ if and only if at some point in computing $zL$, a $\res_p$ is applied to $D_{2m}$-modules with $p\not|m$ or $m=p$. The result follows. 
\end{proof}

\section{Relations between restrictions and inductions between modules over dihedral groups of general order}\label{s4}
We may say a few things about the relations in general algebras $A_{P,\mathcal{M}}$. 

\begin{myprop}
\label{relsprop}
For any set of primes $P$, any $A_P$-submodule $\mathcal{M}\subset\mathcal{G}$, and any $p,q\in P$, the following equalities hold. 
\begin{enumerate}
\item[$($i$)$]
$\varphi_{P,\mathcal{M}}(\res_p\res_q)=\varphi_{P,\mathcal{M}}(\res_q\res_p)$.
\item[$($ii$)$]
$\varphi_{P,\mathcal{M}}(\ind_p\ind_q)=\varphi_{P,\mathcal{M}}(\ind_q\ind_p)$. 

\end{enumerate}
\end{myprop}
\begin{proof}
{\bf Part (i):} In cases where $\res_p$ or $\res_q$ acts by 0, the statement is clear. Assume therefore that this is not the case. The compositions $D_{2n}\hookrightarrow D_{2pn}\hookrightarrow D_{2pqn}$ and $D_{2n}\hookrightarrow D_{2qn}\hookrightarrow D_{2pqn}$ of inclusions as in \eqref{incleq} are the same, hence for any $D_{2pqn}$-module $M$ we have
\begin{equation*}
(M_{|D_{2pn}})_{|D_{2n}}\cong (M_{|D_{2qn}})_{|D_{2n}}
\end{equation*}
so that $\res_p\res_qM=\res_q\res_pM$. Hence $\res_p\res_q-\res_q\res_p$ lies in the kernel of $\varphi_{P,\mathcal{M}}$.  

{\bf Part(ii):} This follows by Frobenius reciprocity. 
\end{proof}

\begin{myprop}
\label{somecommprop}
For any distinct primes $p$ and $q$ and any simple $D_{2pn}$-module $L\in\mathcal{G}$ where $n>1$ (equivalently any simple $D_{2m}$-module such that $\res_pL\ne 0)$, we have that
\begin{equation*}
\res_p\ind_qL=\ind_q\res_pL.
\end{equation*}
\end{myprop}
\begin{proof}
We have that
\begin{align*}
\ind_q\res_p L&\cong \C[D_{2qn}]\otimes_{\C[D_{2n}]}\C[D_{2n}]\otimes_{\C[D_{2n}]}\C[D_{2pn}]\otimes_{\C[D_{2pn}]}L\\&\cong \C[D_{2qn}]\otimes_{\C[D_{2n}]}\C[D_{2pn}]\otimes_{\C[D_{2pn}]}L
\end{align*}
and
\begin{align*}
\res_p\ind_qL&\cong \C[D_{2qn}]\otimes_{\C[D_{2qn}]}\C[D_{2pqn}]\otimes_{\C[D_{2pn}]}\C[D_{2pn}]\otimes_{\C[D_{2pn}]}L\\&\cong \C[D_{2pqn}]\otimes_{\C[D_{2pn}]}L.
\end{align*}
It then suffices to show that the homomorphism of $\C[D_{2qn}]$-$\C[D_{2pn}]$-bimodules
\begin{align*}
f:\C[D_{2qn}]\otimes_{\C[D_{2n}]}\C[D_{2pn}]&\rightarrow \C[D_{2pqn}]\\x\otimes_{\C[D_{2n}]}y&\mapsto xy
\end{align*}
induced by the inclusion \eqref{incleq} is in fact an isomorphism. It is clear from \eqref{incleq} that $f(1\otimes_{\C[D_{2n}]}s_{pn})=s_{pqn}$, that  $f(1\otimes_{\C[D_{2n}]}r_{pn})=r_{pqn}^q$, and that $f(r_{qn}\otimes_{\C[D_{2n}]} 1)=r_{pqn}^p$. Since $p$ and $q$ are distinct primes, the diophantine equation $pu+qv=1$ is solvable in $u$ and $v$, so that we have $f(r_{qn}^u\otimes_{\C[D_{2n}]}r_{pn}^v)=r_{pqn}^{pu+qv}=r_{pqn}$. Since $s_{pqn}$ and $r_{pqn}$ generate $D_{2pqn}$, we get that $f$ is surjective. The module $\C[D_{2qn}]\otimes_{\C[D_{2n}]}\C[D_{2pn}]$ has dimension $2qn\cdot 2pn/(2n)=2pqn$, which is also the dimension of $\C[D_{2pqn}]$. Hence $f$ must indeed be an isomorphism. 
\end{proof}

\section{Results for the algebras \texorpdfstring{$A_{P,\mathcal{M}}$}{APM} for modules over dihedral groups of order not divisible by 4}\label{s5}
Throughout this section, $P$ will be a set of odd prime numbers, and $\mathcal{M}\subset\mathcal{G}$ will be an $A_P$-submodule spanned by simple modules over dihedral groups $D_{2n}$ with $n$ odd. For the main results we will furthermore require that for each $n$, either all simple $D_{2n}$-modules or none lie in $\mathcal{M}$. This latter condition means that $\mathcal{M}$ will consists of entire induction/restriction diagrams as in Figure 1, rather than merely some connected components. Allowing for $\mathcal{M}$ which do not satisfy this condition would give rise to an unwieldy amount of additional cases, although it seems that these too should in principle be amenable to the methods used in the paper. 

The main objective of the present section is to find a basis and a generating set of relations for the algebra $A_{P,\mathcal{M}}$. We will need to consider two cases, depending on whether $\mathcal{M}$ contains a $D_{2n}$-module with all prime factors of $n$ lying in $P$ or not. These cases will be developed in parallel, and culminate in the theorems \ref{babybasisthm} and \ref{basisthm} respectively.

\subsection{A translation of the induction/restriction diagrams}

The Lemma \ref{partfunlem} to follow may seem quite technical, but formalizes something which is fairly easy to corroborate on an intuitive level by looking at the induction/restriction diagram in Figure 1. It morally says the following: Pick any vertex (i.e. simple module) and consider the subdiagram formed by adding all vertices which are both connected to the starting one and also lie at the same level as or higher than the starting one. Then this subdiagram is isomorphic to one of the connected components of entire induction/restriction diagram. It may nevertheless be preferable to skip ahead at this point and refer back to the lemma and its proof when it is used in Proposition \ref{nadirprop} and Lemma \ref{deplem}. The examples \ref{nadirex} and \ref{depex} illustrate the latter results and should shed further light also on the ideas behind Lemma \ref{partfunlem}. 

The following notation for the two-dimensional simple modules over the dihedral groups will prove convenient in the statement of the lemma. Let
\begin{equation*}
\mathcal{I}=\lb 1,p\rb^*
\end{equation*}
be the free monoid generated by $\lb 1,p\rb$. Let $n\ge 3$ be an odd integer, let $k\in \Z$, and set $W_{k}^{()}(n)=W_{k}(n)$. For $I\in\mathcal{I}$, let $\len(I)$ denote the length of the word $I$. For each odd prime $p$, define inductively $W_{k}^I(n)$ for all $I\in\mathcal{I}$ by considering the set
\begin{equation*}
K_{p,n,I}=\{ k'\in\N: 0<k'<\frac{np^{\len(I)+1}}{2}, \res_p(W_{k'}(np^{\len(I)+1}))=W_{k}^I(n)\},
\end{equation*}
by letting $k_{p,n,I,j}$ be the $j$:th smallest element in $K_{p,n,I}$ for each $j\in\lb 1,p\rb$ (this choice of ordering is not essential), and finally defining
\begin{equation*}
W_{k}^{I(j)}(n)=W_{k_{p,n,I,j}}(np^{\len(I)+1}).     
\end{equation*}
\begin{mylem}
\label{partfunlem}
Let $P$ be some set of odd primes, let $p\in P$, and let $\mathcal{M}\subset\mathcal{G}$ be some $A_P$-submodule spanned by simple $D_{n}$-modules, for $n$ odd, and such that for a fixed $n$ either every simple $D_{2n}$-module or none lies in $\mathcal{M}$. Let $J$ range over $\N$. 

Define linear partial functions (i.e. linear maps each defined only on some subspace of its domain) $\Phi_{p,m,J}$ by linearly extending
\begin{align*}
\Phi_{p,m,J}:\mathcal{M}&\rightarrow\mathcal{M}\\
V_{a,b}(mp^{J'})&\mapsto V_{a,b}(mp^{J'+J})\\
W_{k}(mp^{J'})&\mapsto W_{kp^{J}}(mp^{J'+J}),
\end{align*}
for all $k$ divisible by $p$, for all $J'\in \N$ and for all odd $m$ with either $m=p$ or $p\not|m$ such that the modules lie in $\mathcal{M}$. 

Also define, for all $k,k'\in\Z$ and odd $m$ satisfying that $m=p$ or $p\not|m$ and $k\in\lb 1,\frac{m-1}{2}\rb$, linear partial functions $\Psi_{k,m,k',J}$ by linearly extending
\begin{align*}
\Psi_{p,k,m,k',J}:\mathcal{M}&\rightarrow\mathcal{M}\\
W_k^I(m)&\mapsto W_{k'}^I(mp^J)
\end{align*}
for all $I\in\mathcal{I}$. 

The following statements then hold.
\begin{enumerate}
\item[$($i$)$]
For any $D_{2mp^l}$-module $M$ and any $J\in\N$, there is a $\Gamma\in\{\Phi_{p,m,J},\Psi_{p,k,m,k',J}:k,k'\in\Z\}$ such that $M\in\dom(\Gamma)$. 

\item[$($ii$)$]
The domains $\dom(\Phi_{p,m,J})$ and $\dom(\Psi_{p,k,m,k',J})$ are closed under the action of $A_{P,\mathcal{M}}$.

\item[$($iii$)$]
There exist partial linear functions
\begin{equation*}
\Phi_{p,m,J}^{-1}:\mathcal{M}\rightarrow\mathcal{M}
\end{equation*}
and
\begin{equation*}
\Psi_{p,k,m,k',J}^{-1}:\mathcal{M}\rightarrow\mathcal{M}
\end{equation*}
which are partial inverses to $\Phi_{p,m,J}$ and $\Psi_{p,k,m,k',J}$ respectively, i.e. $\Phi_{p,m,J}^{-1}\circ \Phi_{p,m,J}=\id_{\dom(\Phi_{p,m,J})}$, $\Phi_{p,m,J}\circ \Phi_{p,m,J}^{-1}=\id_{\dom(\Phi^{-1}_{p,m,J})}$, $\Psi_{p,k,m,k',J}^{-1}\circ \Psi_{p,k,m,k',J}=\id_{\dom(\Psi_{p,k,m,k',J})}$, and $\Psi_{p,k,m,k',J}\circ \Psi_{p,k,m,k',J}^{-1}=\id_{\dom(\Psi^{-1}_{p,k,m,k',J})}$. 

\item[$($iv$)$]
Let $M\in\mathcal{M}$ be a simple module and $z\in A_{P,\mathcal{M}}$ be a monomial such that $zM\ne 0$.

If $M\in\dom(\Phi_{p,m,J})$, we have
\begin{equation*}
\Phi_{p,m,J}zM=z\Phi_{p,m,J}M.
\end{equation*}

If $M\in\dom(\Psi_{p,k,m,k',J})$, we have
\begin{equation*}
\Psi_{p,k,m,k',J}zM=z\Psi_{p,k,m,k',J}M.
\end{equation*}

\end{enumerate}
\end{mylem}
\begin{proof}
{\bf Part (i):} Every $V_{a,b}(mp^l)$ and every $W_k(mp^l)$ with $p|k$ lies in the domain of some $\Phi_{p,m,J}$, while every $W_k(mp^l)$ with $p\not|k$ lies in the domain of some $\Psi_{p,k,m,k',J}$. 

{\bf Part (ii):} It is clear from Proposition \ref{resindprop} that the set of modules of the forms $V_{a,b}(mp^{J'})$ and $W_k(mp^{J'}$ with $k$ divisible by $p$ is closed under the action of $A_{P,\mathcal{M}}$, and that the same holds for the set of modules of the form $W_k^I(m)$ where $k\in\lb 1,\frac{m-1}{2}\rb$. 

{\bf Part (iii):} The partial functions $\Phi_{p,m,J}$ and $\Psi_{k,m,k',J}$ are clearly injective on their domains, so they have partial inverses with domains $\im(\Phi_{p,m,J})$ and $\im(\Psi_{k,m,k',J})$ respectively. 

{\bf Part(iv):} Let us consider only the case of the partial functions $\Phi_{p,m,J}$ (the proof of the statement for $\Psi_{k,m,k',J}$ is analogous). Let $A,B\in \dom(\Phi_{p,J})$ for some $p$ and $J$. Under the assumption $\res_p(A)\ne 0$, we have that
\begin{equation}
\label{commeq}
\begin{aligned}
1=\dim\Hom(A,\ind_p B)&\Leftrightarrow\\
1=\dim\Hom(\res_p A,B)&\Leftrightarrow\\
1=\dim\Hom(\res_p\Phi_{p,m,J}(A),\Phi_{p,m,J}(B))&\Leftrightarrow\\
1=\dim\Hom(\Phi_{p,m,J}(A),\ind_p\Phi_{p,m,J}(B)).
\end{aligned}
\end{equation}
Here all $\Hom$:s are with respect to the category $\bigoplus_{n\ge 3} D_{2n}\mmod$. We use Frobenius reciprocity for the first and third equivalences. For the second equivalence, we can use Proposition \ref{resindprop} to verify for every pair of simple modules $A$ and $B$ that $B$ occurs as a summand of $\res_p(A)$ if and only if $\Phi_{p,m,J}(B)$ occurs a a summand of $\res_p(\Phi_{p,m,J}(A))$. For instance, $V_{1,1}(mp^2)$ is a summand of $\res_p(W_{mp^2}(mp^3)\cong V_{1,1}(mp^2)\oplus V_{1,-1}(mp^2)$, and indeed $\Phi_{p,m,1}(V_{1,1}(mp^2))=V_{1,1}(mp^3)$ is a summand of $\res_p(\Phi_{p,m,1}(W_{mp^2}(mp^3))=\res_p(W_{mp^3}(mp^4))\cong V_{1,1}(mp^3)\oplus V_{1,-1}(mp^3)$. 

Because the dimensions of $\Hom(\res_pA,B)$ for various modules $B$ encode the result of applying $\res_p$ to $A$, because the dimensions of $\Hom(A,\ind_pB)$ for various modules $A$ encode the result of applying $\ind_p$ to $B$, and because part (i) ensures that $A,B\in\dom(\Phi_{p,m,J})$ causes no restriction in this encoding, we get that the desired result follows for $z$ of degree 1. From this, the more general result is immediate. 
\end{proof}

\begin{myprop}
\label{nadirprop}
Let $P$ be a set of odd primes, and let $\mathcal{M}\subset\mathcal{G}$ be some $A_P$-submodule spanned by simple $D_{2n}$-modules with $n$ odd. Let $z_1\in A_P$ be a monomial, and let $z_2$ be the result of reordering the factors of $z_1$ in a way such that the relative order of factors $\res_p$ and $\ind_p$ for each fixed $p\in P$ is unchanged. Assume that at least one of the following conditions holds
\begin{enumerate}
\item
Either none or both of $z_1$ and $z_2$ have a total nadir.

\item
There is no simple $D_{2n}$-module in $\mathcal{M}$ such that all prime factors of $n$ belong to $P$. 
\end{enumerate}
Then
\begin{equation*}
\varphi_{P,\mathcal{M}}(z_1)=\varphi_{P,\mathcal{M}}(z_2).
\end{equation*}
\end{myprop}
\begin{proof}
By Lemma \ref{wloglem} we may assume that $\mathcal{M}$ satisfies the assumptions of Lemma \ref{partfunlem}. Let $L\in\mathcal{M}$ be any simple $D_{2n}$-module. We have by Lemma \ref{nadirlem} that either $z_1L=0=z_2L$ or $z_1L\ne 0\ne z_2L$ (here we must use either of the two assumptions). In the first case we are done, so let us consider the latter case. If the second assumption holds, the result follows immediately from Propositions \ref{relsprop} and \ref{somecommprop}. If at least one of the assumptions holds, the following argument applies. By Lemmata \ref{nadirlem} and \ref{partfunlem} we may (using the notation of the latter lemma) pick a sequence of
\begin{equation*}
\Gamma_i\in\{\Phi_{p,m,J},\Psi_{k,m,k',J}|p,m,J,k,k'\text{ ranging over all possibilities allowed by Lemma \ref{partfunlem}}\}
\end{equation*}
with the index $i$ ranging from 1 to some positive integer $l$, and furthermore a partial inverse $\Gamma_i^{-1}$ of each $\Gamma_i$ such that $z'\Gamma_1\circ\dots\circ\Gamma_lL$ and is well-defined and nonzero for all possible results $z'$ of reordering the factors of $z_1$. By Proposition \ref{somecommprop} we have for $p,q\in P$ distinct that
\begin{equation*}
\res_p\ind_qL'=\ind_q\res_pL'
\end{equation*}
for any modules $L'$ such that $\res_pL'\ne 0$. Also using Proposition \ref{relsprop}, we by Lemma \ref{partfunlem} and the choice of our $\Gamma_i$ then have
\begin{align*}
z_1L&=\Gamma_l^{-1}\circ\dots\circ\Gamma_1^{-1}\circ\Gamma_1\circ\dots\circ\Gamma_lz_1L\\&=\Gamma_l^{-1}\circ\dots\circ\Gamma_1^{-1}z_1\Gamma_1\circ\dots\circ\Gamma_lL\\&=\Gamma_l^{-1}\circ\dots\circ\Gamma_1^{-1}z_2\Gamma_1\circ\dots\circ\Gamma_lL\\&=\Gamma_l^{-1}\circ\dots\circ\Gamma_1^{-1}\circ\Gamma_1\circ\dots\circ\Gamma_lz_2L=z_2L.
\end{align*}
\end{proof}

\begin{myex}
\label{nadirex}
Neither of the monomials $z_1=\ind_5\res_5\ind_3\res_3$ and $z_2=\ind_3\res_3\ind_5\res_5$ has a total nadir, so according to Proposition \ref{nadirprop}, we have $\varphi_{P,\mathcal{M}}(z_1)=\varphi_{P,\mathcal{M}}(z_2)$ for any $P$ containing $3$ and $5$, and $\mathcal{M}\subset\mathcal{G}$ as in the proposition statement. In particular, we should expect
\begin{equation*}
z_1V_{1,1}(15)=z_2V_{1,1}(15).
\end{equation*}
This is indeed the case, as confirmed by the following computations:
\begin{align*}
&z_1V_{1,1}(15)=\ind_5\res_5\ind_3\res_3V_{1,1}(15)=\ind_5\res_5\ind_3V_{1,1}(5)\\&=\ind_5\res_5(V_{1,1}(15)\oplus W_{5}(15))=\ind_5(V_{1,1}(3)\oplus W_2(3))\\&=V_{1,1}(15)\oplus W_3(15)\oplus W_6(15)\oplus W_2(15)\oplus W_1(15)\oplus W_5(15)\oplus W_4(15)\oplus W_7(15),
\end{align*}
and 
\begin{align*}
&z_2 V_{1,1}(15)=\ind_3\res_3\ind_5\res_5 V_{1,1}(15)=\ind_3\res_3\ind_5 V_{1,1}(3)\\&=\ind_3\res_3(V_{1,1}(15)\oplus W_3(15)\oplus W_6(15))=\ind_3(V_{1,1}(5)\oplus W_2(5)\oplus W_1(5))\\&=V_{1,1}(15)\oplus W_5(15)\oplus W_2(15)\oplus W_3(15)\oplus W_7(15)\oplus W_1(15)\oplus W_4(15)\oplus W_6(15),
\end{align*}
which have equal results by direct comparison. 

It is not, however, the case that $z_1V_{1,1}(15)$ is invariant under elementary transpositions of the factors of $z_1$, even when the composition of these transpositions take $z_1$ to $z_2$. Indeed, we have such transpositions
\begin{align*}
&z_1=\ind_5\res_5\ind_3\res_3\rightsquigarrow \ind_5\ind_3\res_5\res_3\rightsquigarrow \ind_3\ind_5\res_5\res_3\\&\rightsquigarrow \ind_3\ind_5\res_3\res_5\rightsquigarrow \ind_3\res_3\ind_5\res_5=z_2
\end{align*}
but
\begin{equation*}
\ind_3\ind_5\res_3\res_5 V_{1,1}(15)=\ind_3\ind_5\res_3 V_{1,1}(3)=\ind_3\ind_5 0=0\ne z_1 V_{1,1}(15).
\end{equation*}
This is because $\ind_3\ind_5\res_3\res_5$ has a total nadir, so we can not apply Proposition \ref{somecommprop}. In order to circumvent this problem in the proof of Proposition \ref{nadirprop}, we for the ``missing link'' $\ind_3\ind_5\res_3\res_5$ instead compute
\begin{align*}
&\Phi_{3,5,1}^{-1}\ind_3\ind_5\res_3\res_5\Phi_{3,5,1}V_{1,1}(15)=\Phi_{3,5,1}^{-1}\ind_3\ind_5\res_3\res_5 V_{1,1}(45)\\&=\Phi_{3,5,1}^{-1}\ind_3\ind_5 V_{1,1}(3)=\Phi_{3,5,1}^{-1}\ind_3(V_{1,1}(15)\oplus W_3(15)\oplus W_6(15))\\&=\Phi_{3,5,1}^{-1}(V_{1,1}(45)\oplus W_{15}(45)\oplus W_3(45)\oplus W_{12}(45)\oplus W_{18}(45)\oplus W_6(45)\oplus W_9(45)\oplus W_{21}(45))\\&=V_{1,1}(15)\oplus W_5(15)\oplus W_1(15)\oplus W_4(15)\oplus W_6(15)\oplus W_2(15)\oplus W_3(15)\oplus W_7(15).
\end{align*}
This is indeed equal to $z_1V_{1,1}(15)$ and $z_2V_{1,1}(15)$. 
\end{myex}

\begin{mylem}
\label{deplem}
Let $P$ be some set of odd primes and let $\mathcal{M}\subset\mathcal{G}$ be any $A_{P,\mathcal{G}}$-submodule spanned by $D_{2n}$-modules with odd $n$, such that for each fixed $n$ either all or none of the simple $D_{2n}$-modules lie in $\mathcal{M}$. Let also $S\subset A_{P}$ be a set of monomials whose image in $A_{P,\mathcal{M}}$ is linearly dependent and minimal with this property. Then the following hold.
\begin{enumerate}
\item[$($i$)$]
The respective termini of the elements in $S$ with respect to each prime in $P$ are equal, and the respective nadirs of the elements in $S$ with respect to each prime in $P$ are equal. 
\item[$($ii$)$]
If in addition $\mathcal{M}$ consists of $D_{2n}$-modules ($n$ is not necessarily fixed) with all prime factors of $n$ belonging to $P$, then either all elements in $S$ have a total nadir or none has.
\end{enumerate}
\end{mylem}
\begin{proof}
Let $\gamma\in A_P$ be a nonzero linear combination of elements in $S$ whose image in $A_{P,\mathcal{M}}$ is zero. 

{\bf Part (i):} Let $M\in\mathcal{M}$ be an arbitrary simple $D_{2n}$-module. By assumption, we have $\gamma(M)=0$. Fix an arbitrary $p\in P$, let $e_p$ be the largest terminus and $d_p$ be the largest nadir of any element in $S$ with respect to $p$. Write $\gamma=\alpha'+\beta'$, where the termini of the terms in $\alpha'$ with respect to $p$ equal $e_p$ while the termini of the terms in $\beta'$ with respect to $p$ are less than $e_p$. Then $\alpha'(M)$ is a linear combination of $D_{2np^{e_p}}$-modules (or trivially zero if $2np^{e_p}$ is not an integer $\ge 3$) while $\beta'(M)$ is a linear combination of modules over other groups, so we must have $\alpha'(M)=0$. Hence by the minimality of $S$, there are no monomial terms in $\gamma$ except for those occurring already in $\alpha'$, so we must have $\beta'=0$ and $\gamma=\alpha'$.  

As for the corresponding statement for nadirs, write $n=mp^{J_1}$, where either $p\not|m$ or $m=p$. Write also $\gamma=\alpha+\beta$, where the nadirs of the terms in $\alpha$ with respect to $p$ equal $d_p$ while the nadirs of the terms of $\beta$ with respect to $p$ are less than $d_p$. It follows from our assumptions that $\alpha\ne 0$. Assume towards a contradiction that also $\beta\ne 0$. 

Since $M$ is arbitrary, it now suffices to show that also $\alpha(M)=0$ in order to contradict the minimality of $S$. We may without loss of generality assume that $J_1\ge -d_p$, since otherwise $\alpha(M)=0$ trivially. If $\res_p^{-d_p}(M)$ has nonzero projection onto some $V_{a,b}(m')$, then $M$ is either of the form $M=V_{a,b}(mp^{J_1})$ or of the form $M=W_k(mp^{J_1})$ for some $k$ divisible by $mp^{J_1+d_p}$. Define in this case $\Gamma=\Phi_{p,m,J_1+d_p}$. Then either $\Gamma(V_{a,b}(mp^{-d_p}))=M$ or $\Gamma(W_{kp^{-J_1-d_p}}(mp^{-d_p}))=M$.

If instead $\res_p^{-d_p}(M)$ has zero projection onto all one-dimensional simple dihedral group modules, then $M$ may be written on the forms $M=W_k(mp^{J_1})=W_{k'}^{I\cdot I'}(m')$ with $\res_p^{-d_p}(M)=W_k(mp^{J_1+d_p})=W_{k'}^{I}(m')$ simple. Define in this case $\Gamma=\Psi_{p,1,m,k,J_1+d_p}$. Then $\Gamma(W_1^{I'}(m))=M$. 

Let $\Gamma^{-1}$ be the partial inverse of $\Gamma$ (the existence of which is given by part (iii) of Lemma \ref{partfunlem}). In any of the above cases we have that $M\in\im(\Gamma)$, so that $M\in\dom(\Gamma^{-1})$. Let $M'=\Gamma^{-1}(M)$. Using the assumption on $\gamma$, we compute
\begin{equation*}
0=\gamma(M')=(\alpha+\beta)(M')=\alpha(M'),
\end{equation*}
where for the last equality we used Lemma \ref{nadirlem} together with the facts that the nadirs of the terms of $\beta$ with respect to $p$ are smaller than $d_p$ and $M'$ is a $D_{2mp^{-d_p}}$-module. This implies that indeed 
\begin{equation*}
0=\Gamma(\alpha(M'))=\alpha(\Gamma(M'))=\alpha(M),
\end{equation*}
where for the second equality we used part (iv) of Lemma \ref{partfunlem} together with the fact that no monomial term in $\alpha$ annihilates $M'$. This latter fact in turn follows from $M'$ being a $D_{2mp^{-d_p}}$-module, the terms of $\alpha$ having nadir $d_p$ with respect to $p$, and Lemma \ref{nadirlem}. 

{\bf Part (ii):} Let $M\in\mathcal{M}$ be an arbitrary simple $D_{2n}$-module (where we have assumed that the prime factors of $n$ all belong to $P$). By our assumption on $\gamma$, we have $\gamma(M)=0$. Write $\gamma=\alpha+\beta$ where this time all monomial terms in $\alpha$ do not have any total nadir and all monomial terms in $\beta$ do have a total nadir. Assume towards a contradiction that $\alpha\ne 0\ne\beta$ (that $\alpha\ne 0$ implies in particular that $|P|\ge 2$). By part (i) we may assume that for each $p\in P$, the same number $d_p$ is the nadir of every monomial term of $\gamma$ with respect to $p$. We may also assume as in part (i) that for every $p\in P$ we have $np^{d_p}$ is an integer $\ge 3$. In particular, we may assume that $n$ is not a prime power, since otherwise $\alpha(M)=0$ already because the terms of $\alpha$ must have nonzero nadirs with respect to at least two different primes. 

Fix any total order on $P$, and an indexing such that for $p_i,p_j\in P$ we have $p_i<p_j$ if and only if $i<j$, where $i,j\in\lb 1,|P|\rb$. Define some partial functions $\Gamma_i$ and modules $M_i$ as follows. Let first $M_0=M$. Then inductively let $\Gamma_{i+1}$ be constructed out of $M_i$ and $p_{i+1}$ as $\Gamma$ was constructed out of $M$ and $p$ in part (i), and set $M_{i+1}=\Gamma_{i+1}^{-1}(M_i)$, up to $i=|P|-1$. If $M_i$ is a $D_{2n_i}$-module, note that $n_i$ will contain a factor $p^{-d_p}$ for every $p\in P$. Since $d_p\ne 0$ for at least two choices of $p\in P$, no $n_i$ is a prime power. It in particular follows that the $m$-value in each step $i$ of the construction will satisfy $p_i\not|m$, hence that the finally obtained module $M_{|P|}$ is a $D_{2\prod_{p\in P}p^{-d_p}}$-module. Similarly to part (i), we now compute 
\begin{equation*}
0=\gamma(M_{|P|})=(\alpha+\beta)(M_{|P|})=\alpha(M_{|P|}),
\end{equation*}
where for the last equality we used part (ii) of Lemma \ref{nadirlem} together with the facts that the terms of $\beta$ have a total nadir and $M_{|P|}$ is a $D_{2\prod_{p\in P}p^{-d_p}}$-module. This as before implies that
\begin{equation*}
0=\Gamma_1\circ\dots\circ\Gamma_{|P|}(\alpha(M_{|P|}))=\alpha(\Gamma_1\circ\dots\circ\Gamma_{|P|}(M_{|P|}))=\alpha(M),
\end{equation*}
where for the second equality we used part (iv) of Lemma \ref{partfunlem} together with the fact that no monomial term of $\alpha$ annihilates $M_{|P|}$. This latter fact in turn follows from $M_{|P|}$ being a $D_{2\prod_{p\in P}p^{-d_p}}$-module, the terms of $\alpha$ having no total nadir, and Lemma \ref{nadirlem}.
\end{proof}

\begin{myex}
\label{depex}
Consider the following situation, which is a very special case of the proof of part (i) of Lemma \ref{deplem}. Let $P$ be a set of odd primes with $3\in P$, let $\alpha,\beta\in A_P$ be linear combinations of monomials such that the nadir of every monomial term of $\alpha$ with respect to $3$ is $-1$ while the nadir of every monomial term of $\beta$ with respect to $3$ is $-2$. Assume that
\begin{equation*}
(\alpha+\beta)W_{-1+5}(15)=0.
\end{equation*}
We will show that
\begin{equation*}
\alpha W_{-5+15}(45)=0.
\end{equation*}
Note that the action of $\res_3$ maps the modules $W_5(45)$, $W_{-5+15)}(45)$ and $W_{5+15}(45)$ to $W_{5}(15)$. Hence $W_{-5+15}(45)=W_5^{(2)}(15)$. Similarly, $W_{-1+5}(15)=W_1^{(2)}(5)$. Because every monomial term of $\beta$ must annihilate $W_{-1+5}(15)$ by Lemma \ref{nadirlem}, we have
\begin{align*}
0=(\alpha+\beta)W_{-1+5}(15)=\alpha W_{-1+5}(15),
\end{align*}
from which follows that
\begin{align*}
0=\Psi_{3,1,5,5,1}\alpha W_{-1+5}(15)=\alpha\Psi_{3,1,5,5,1}W_{-1+5}(15)=\alpha W_{-5+15}(45).
\end{align*}
\end{myex}

\begin{mycor}
\label{welldefcor}
Let $P$ be some set of odd primes and let $\mathcal{M}\subset\mathcal{G}$ be some $A_P$-submodule spanned by simple $D_{2n}$-modules with $n$ odd and such that for each fixed $n$ either all or none of the simple $D_{2n}$-modules lie in $\mathcal{M}$. Furthermore, let $p\in P$ be arbitrary, and let $z_1,z_2\in A_P$ be monomials such that $\varphi_{P,\mathcal{M}}(z_1)=\varphi_{P,\mathcal{M}}(z_2)$. Then the respective termini and nadirs of $z_1$ and $z_2$ are equal.
\end{mycor}
\begin{proof}
If $\varphi_{P,\mathcal{M}}(z_1)=\varphi_{P,\mathcal{M}}(z_2)$, then $z_1-z_2\in\ker(\varphi_{P,\mathcal{M}})$. Now apply Lemma \ref{deplem} to $S=\{z_1,z_2\}$. 
\end{proof}
For an image, $z\in A_{P,\mathcal{M}}$, of a monomial in $A_P$, we define termini and nadirs of $z$ to be those of a monomial representative in $A_P$. By Corollary \ref{welldefcor}, these are well-defined.

\subsection{Additional relations of \texorpdfstring{$A_{P,\mathcal{M}}$}{APM}}

\begin{mylem}
\label{rellem}
Let $P$ be a set of odd primes, and let $\mathcal{M}\subset\mathcal{G}$ be some $A_P$-submodule spanned by simple $D_{2n}$-modules with $n$ odd. Let $p\in P$. Then 
\begin{equation*}
\varphi_{P,\mathcal{M}}(\res_p\res_p\ind_p\ind_p)=\varphi_{P,\mathcal{M}}((p+1)\res_p\ind_p-p).
\end{equation*}
\end{mylem}
\begin{proof}
By direct computation using Proposition \ref{resindprop} (or by looking at Figure 1) we have for any $n\ge 3$ that
\begin{align*}
(p -(p+1)\res_p\ind_p+\res_p\res_p\ind_p\ind_p)(W_k(n))&=pW_k(n)-(p+1)pW_k(n)+p^2W_k(n)\\&=0,
\end{align*}
and
\begin{align*}
&(p -(p+1)\res_p\ind_p+\res_p\res_p\ind_p\ind_p)(V_{a,b}(n))\\&=pV_{a,b}(n)-(p+1)(V_{a,b}(n)+\frac{p-1}{2}(V_{a,b}(n)+V_{-a,b}(n)))+V_{a,b}(n)\\&+\frac{p-1}{2}(V_{a,b}(n)+V_{-a,b}(n))+p\frac{p-1}{2}(V_{a,b}(n)+V_{-a,b}(n))=0.
\end{align*}
The desired result follows.
\end{proof}

\begin{mylem}
\label{mixrel}
Let $P$ be a set of odd primes, and let $\mathcal{M}\subset\mathcal{G}$ be some $A_P$-submodule spanned by simple $D_{2n}$-modules with $n$ odd. Let $p,q\in P$. Then
\begin{equation*}
\varphi_{P,\mathcal{M}}(\frac{1}{p-1}(\res_p\ind_p-p))=\varphi_{P,\mathcal{M}}(\frac{1}{q-1}(\res_q\ind_q-q))
\end{equation*}
\end{mylem}
\begin{proof}
For $n\ge 3$, we have 
\begin{equation*}
(\res_p\ind_p-p)W_k(n)=0=(\res_q\ind_q-q)W_k(n).
\end{equation*}
Also 
\begin{equation*}
\frac{1}{p-1}(\res_p\ind_p-p)V_{1,b}(n)=\frac{-V_{1,b}(n)+V_{1,-b}(n)}{2}=\frac{1}{q-1}(\res_q\ind_q-q)V_{1,b}(n). 
\end{equation*}
\end{proof}
The following is a corollary of Proposition \ref{relsprop}, Proposition \ref{nadirprop} and Lemma \ref{mixrel}. Note that we from here on will abuse notation and relax the distinction between $\res_p$ and $\ind_p$ on one hand and their images under $\varphi_{P,\mathcal{M}}$ on the other.
\begin{mycor}
\label{ccor}
Let $P$ be a set of odd primes, let $p\in P$, and let $\mathcal{M}\subset\mathcal{G}$ be any $A_{P,\mathcal{G}}$-submodule spanned by $D_{2n}$-modules with odd $n$. Then the center of $A_{P,\mathcal{M}}$ contains the element $\res_p\ind_p$.
\end{mycor}
\begin{proof}
We may without loss of generality assume that $P$ contains all odd primes. It suffices to show that $\res_p\ind_p$ commutes with $\res_q$ and $\ind_q$ for all primes $q$. If $q=p$ we may pick any odd prime $p'\ne p$ and first use Lemma \ref{mixrel} to rewrite 
\begin{equation*}
\res_p\ind_p=\frac{p-1}{p'-1}(\res_{p'}\ind_{p'}-p')+p
\end{equation*}
Thus we may assume that $q\ne p$. It is clear that $\res_q\res_p\ind_p$, $\res_p\ind_p\res_q$, $\ind_q\res_p\ind_p$ and $\res_p\ind_p\ind_q$ all have total nadirs (note that we may ignore the factor $\res_p\ind_p$ when determining whether an initial subword is a total nadir in some other; for instance, $\res_q$ and $\res_p\ind_p\res_q$ are both total nadirs in $\res_p\ind_p\res_q$), so the result now follows from Proposition \ref{nadirprop}. 
\end{proof}

\begin{mylem}
\label{nadirendlem}
Let $P$ be some set of odd primes and let $\mathcal{M}\subset\mathcal{G}$ be any $A_{P,\mathcal{G}}$-submodule spanned by $D_{2n}$-modules with odd $n$, such that for each fixed $n$ either all or none of the simple $D_{2n}$-modules lie in $\mathcal{M}$. Let $p\in P$, and let $z\in A_{p,\mathcal{M}}\subset A_{P,\mathcal{M}}$. Let further $d_p$ be the nadir and $e_p$ be the terminus of $z$ with respect to $p$. Let $p_1\in P$ be arbitrary. Then the following hold.
\begin{enumerate}
\item[$($i$)$]
If $d_p=0$, then $z$ may be written as a linear combination of monomials of the form 
\begin{equation*}
(\res_{p_1}\ind_{p_1})^t\ind_p^l
\end{equation*}
with $t\in\{ 0,1\}$ and $l\in\N$. In particular, the empty subword is a nadir in $z$ with respect to $p$.
\item[$($ii$)$]
If $d_p=e_p$, then $z$ may be written as a linear combination of monomials of the form 
\begin{equation*}
(\res_{p_1}\ind_{p_1})^t\res_p^k
\end{equation*}
with $t\in\{ 0,1\}$ and $k\in\N$. In particular, $z$ is a nadir in $z$ with respect to $p$. 
\item[$($iii$)$]
If $0\ne d_p\ne e_p$, then $z$ may be written as a linear combination of monomials of the form 
\begin{equation*}
(\res_{p_1}\ind_{p_1})^t\ind_p^l\res_p^k
\end{equation*}
with $k,l\in\Z_{>0}$. In particular, neither the empty subword nor $z$ is a nadir in $z$ with respect to $p$. 
\end{enumerate}
\end{mylem}
\begin{proof}
The result follows immediately from Lemma \ref{rellem}, Lemma \ref{mixrel}, and Corollary \ref{ccor}. 
\end{proof}

\subsection{A basis for \texorpdfstring{$A_{P,\mathcal{M}}$}{APM}}

We may now describe a basis for our algebra $A_{P,\mathcal{M}}$. Following Theorem \ref{basisthm} is an example which illustrates the proof in some very specific cases.
\begin{mythm}
\label{babybasisthm}
Let $P$ be a set of odd primes. Let also $\mathcal{M}\subset\mathcal{G}$ be some $A_P$-submodule spanned by simple $D_{2n}$-modules with $n$ odd and such that for each fixed $n$ either all or none of the simple $D_{2n}$-modules lie in $\mathcal{M}$, and furthermore such that there is no simple $D_{2n}$-module in $\mathcal{M}$ with all prime factors of $n$ belonging to $P$. Fix any total order on $P$ (say the restriction of the usual one on $\N$) and index the elements of $P$ by $p_i<p_j$ with $i,j\in\lb 1,|P|\rb$ if and only if $i<j$. Then the monomials of the forms
\begin{equation*}
(\res_{p_1}\ind_{p_1})^t\ind_{p_1}^{l_1}\dots\ind_{p_{|P|}}^{l_{|P|}}\res_{p_1}^{k_1}\dots\res_{p_{|P|}}^{k_{|P|}}
\end{equation*}
with $t\in\{ 0,1\}$ and $k_i,l_i\in\N$ form a basis of $A_{P,\mathcal{M}}$, and the relations of $A_{P,\mathcal{M}}$ are generated by the following ones.
\begin{enumerate}
\item[$($i$)$]
$\res_p\res_q=\res_q\res_p$.
\item[$($ii$)$]
$\ind_p\ind_q=\ind_q\ind_p$. 
\item[$($iii$)$]
$\ind_q\res_p=\res_p\ind_q$,\\
for $p\ne q$.
\item[$($iv$)$]
$\res_p\res_p\ind_p\ind_p=(p+1)\res_p\ind_p-p$.
\item[$($v$)$]
$\frac{1}{p-1}(\res_p\ind_p-p)=\frac{1}{q-1}(\res_q\ind_q-q)$.
\end{enumerate}
\end{mythm}
\begin{proof}
That the monomials $(\res_{p_1}\ind_{p_1})^t\ind_{p_1}^{l_1}\dots\ind_{p_{|P|}}^{l_{|P|}}\res_{p_1}^{k_1}\dots\res_{p_{|P|}}^{k_{|P|}}$ span $A_{P,\mathcal{M}}$ follows readily from Proposition \ref{nadirprop} and Lemma \ref{nadirendlem}. For the proof of linear independence, we refer to the first part of the proof of linear independence for Theorem \ref{basisthm}, which applies mutatis mutandis here too (note that while the proof of Theorem \ref{basisthm} refers to the present theorem, it does so only in the final paragraph of the linear independence proof, so there is no circularity). Since every $z\in A_{P,\mathcal{M}}$ can be written as a linear combination of the basis elements using the relations (i)-(v) (via Proposition \ref{nadirprop}, Lemma \ref{rellem}, Corollary \ref{ccor}, and Lemma \ref{nadirendlem}), these relations indeed generate all relations of $A_{P,\mathcal{M}}$. 
\end{proof}
\begin{mythm}
\label{basisthm}
Let $P$ be a set of odd primes. Let also $\mathcal{M}\subset\mathcal{G}$ be some $A_P$-submodule spanned by simple $D_{2n}$-modules with $n$ odd and such that for each fixed $n$ either all or none of the simple $D_{2n}$-modules lie in $\mathcal{M}$, and furthermore such that there is a simple $D_{2n}$-module in $\mathcal{M}$ with all prime factors of $n$ belonging to $P$. Fix any total order on $P$ (say the restriction of the usual one on $\N$) and index the elements of $P$ by $p_i<p_j$ with $i,j\in\lb 1,|P|\rb$ if and only if $i<j$. Then the monomials of the forms
\begin{enumerate}
\item[$($i$)$]
$(\res_{p_1}\ind_{p_1})^t\ind_{p_1}^{l_1}\dots\ind_{p_{|P|}}^{l_{|P|}}\res_{p_1}^{k_1}\dots\res_{p_{|P|}}^{k_{|P|}}$\\
with $t\in\{ 0,1\}$, and $k_i,l_i\in\N$,
\item[$($ii$)$]
$(\res_{p_1}\ind_{p_1})^t\ind_{p_1}^{l_1}\res_{p_1}^{k_1}\dots\ind_{p_{|P|}}^{l_{|P|}}\res_{p_{|P|}}^{k_{|P|}}$\\
with $t\in\{ 0,1\}$, and $k_i,l_i\in\N$ such that $k_i\ne 0\ne l_i$ for at least two choices of $i$,
\item[$($iii$)$]
$(\res_{p_1}\ind_{p_1})^t\res_{p_i}^k\ind_{p_j}^l$\\
with $t\in\{ 0,1\}$, with $i\ne j$, and $k,l\in\Z_{>0}$,
\item[$($iv$)$]
$(\res_{p_1}\ind_{p_1})^t\res_{p_{i\text{ $($mod }|P|)+1}}\ind_{p_i}^l\res_{p_i}^k\ind_{p_{i\text{ $($mod }|P|)+1}}$\\
with $t\in\{ 0,1\}$, and $k,l\in\Z_{>0}$,
\item[$($v$)$]
$(\res_{p_1}\ind_{p_1})^t\ind_{p_j}^l\res_{p_j}^k\ind_{p_1}^{l_1}\dots\ind_{p_{|P|}}^{l_{|P|}}$\\
with $t\in\{ 0,1\}$, with $j\in\lb 1,|P|\rb$, with $k,l\in\Z_{>0}$, and $l_i\in\N$ such that $l_j=0$ but $l_i\ne 0$ for at least one $i$,
\item[$($vi$)$]
$(\res_{p_1}\ind_{p_1})^t\res_{p_1}^{k_1}\dots\res_{p_{|P|}}^{k_{|P|}}\ind_{p_j}^l\res_{p_j}^k$\\
with $t\in\{ 0,1\}$, with $j\in\lb 1,|P|\rb$, with $k,l\in\Z_{>0}$, and $k_i\in\N$ such that $k_j=0$ but $k_i\ne 0$ for at least one $i$,
\end{enumerate}
form a basis of $A_{P,\mathcal{M}}$, and the relations of $A_{P,\mathcal{M}}$ are generated by the following ones.
\begin{enumerate}
\item[$($i$)$]
$\res_p\res_q=\res_q\res_p$.
\item[$($ii$)$]
$\ind_p\ind_q=\ind_q\ind_p$. 
\item[$($iii$)$]
$z_1=z_1$,\\
where $z_2$ is the result of reordering the factors of $z_1$ in a way such that the relative order of factors $\res_p$ and $\ind_p$ for each fixed $p\in P$ is unchanged, and where either both or none of $z_1$ and $z_2$ has a total nadir. 
\item[$($iv$)$]
$\res_p\res_p\ind_p\ind_p=(p+1)\res_p\ind_p-p$.
\item[$($v$)$]
$\frac{1}{p-1}(\res_p\ind_p-p)=\frac{1}{q-1}(\res_q\ind_q-q)$.
\end{enumerate}
\end{mythm}
\begin{proof}
We will use Lemma \ref{nadirendlem} and also use the notation $d_p$ and $e_p$ from that lemma. Let us first show that an arbitrary monomial $z\in A_{P,\mathcal{M}}$ can be written as a linear combination of monomials of forms in the theorem statement. For $i\in \lb 1,|P|\rb$, let $z_i$ be the maximal subword of $z$ consisting entirely of factors $\res_{p_i}$ and $\ind_{p_i}$. 

First consider the case when there is a total nadir in $z$. We may then write $z=z'z''$, where $z''$ is a total nadir in $z$. Let $z'_i$ be the maximal subword of $z'$ consisting entirely of factors $\res_{p_i}$ and $\ind_{p_i}$, and $z''_i$ be the maximal subword of $z''$ consisting entirely of factors $\res_{p_i}$ and $\ind_{p_i}$. By Proposition \ref{nadirprop} we have 
\begin{equation*}
z=z'_1\dots z'_{|P|}z''_1\dots z''_{|P|}.
\end{equation*}
Now apply Lemma \ref{nadirendlem} to write each $z'_i$ as a linear combination of monomials of the form $(\res_{p_1}\ind_{p_1})^t\ind_{p_i}^l$ with $t\in\{0,1\}$ and $l\in\N$ depending on $i$, and also write each $z''_i$ as a linear combination of monomials of the form $(\res_{p_1}\ind_{p_1})^t\res_{p_i}^k$ with $t\in\{0,1\}$ and $k\in\N$ depending on $i$. Now apply Lemma \ref{rellem} and Corollary \ref{ccor} to see that $z$ may be written as a linear combination of monomials of the form (i).

Next we consider several cases where there is no total nadir in $z$. Note that this in particular may only be the case if $|P|>1$. Since there is no total nadir in $z$, we can not have that $d_p=0$ for all $p\in P$, and neither that $d_p=e_p$ for all $p\in P$. Let us say that $z$ \emph{starts on a $p$-edge} if $d_p\ne 0$ while $d_q=0$ for $q\ne p$, and let us say that $z$ \emph{ends on a $p$-edge} if $d_p\ne e_p$ while $d_q=e_q$ for $q\ne p$. 

First consider the case where $z$ neither starts nor ends on a $p$-edge for any $p\in P$. Then by part (iii) of Lemma \ref{nadirendlem}, we have for at least two values of $i$ that there is a nadir in $z_i$ with respect to $p$ which equals neither the empty word nor the entire $z_i$. Similarly to above we then have by Proposition \ref{nadirprop} that
\begin{equation*}
z=z_1\dots z_{|P|},
\end{equation*}
where the $z_i$ may be written as a linear combination of monomials as in parts (i)-(iii) of Lemma \ref{nadirendlem}, with part (iii) being the case for at least two different values of $i$. As above, apply Lemma \ref{rellem} and Corollary \ref{ccor} to see that $z$ may be written as a linear combination of monomials of the form (ii) in the theorem statement.  

Next consider the case where $z$ starts on a $p_i$-edge and ends on a $p_j$-edge, where $i\ne j$. Then by parts (i) and (ii) of Lemma \ref{nadirendlem}, we have that $z_j$ is not a nadir in $z_j$ with respect to $p_j$ and that the empty subword is not a nadir in $z_i$ with respect to $p_i$. Thus we have by Proposition \ref{nadirprop} that
\begin{equation*}
z=z'z_iz_j,
\end{equation*}
where $z'$ is any rearrangement of the factors in $z$ that do not lie in $z_i$ or $z_j$. Note that by Lemma \ref{nadirendlem} we have that $z'$ can be written as a linear combination of elements of the form $(\res_{p_1}\ind_{p_1})^t$ with $t\in\{0,1\}$, that $z_i$ can be written as a linear combination of elements of the form $(\res_{p_1}\ind_{p_1})^t\res_{p_i}^k$ with $t\in\{0,1\}$ and $k$ depending on $i$, and finally that $z_j$ can be written as a linear combination of elements of the form $(\res_{p_1}\in_{p_1})^t\ind_{p_j}^l$ with $t\in\{0,1\}$ and $l$ depending on $j$. As before, apply Lemma \ref{rellem} and Corollary \ref{ccor} to see that $z$ may be written as a linear combination of monomials of the form (iii) in the theorem statement. 

Next consider the case where $z$ starts and ends on a $p_i$-edge. Let $z=z'z''$ where $z''$ is a nadir in $z$ with respect to $p_i$. Then the empty subword and $z$ are both nadirs in $z$ with respect to any $p_j$ with $j\ne i$. Because there is no total nadir in $z$, we must then have that $z''$ contains a factor $\ind_q$ for some $q\ne p_i$, and that $z'$ contains the factor $\res_q$. 
Using Proposition \ref{nadirprop}, Lemma \ref{rellem} and Corollary \ref{ccor} we may write $z$ as a linear combination of monomials of the form $(\res_{p_1}\ind_{p_1})^t\res_q\ind_p^l\res_p^k\ind_q$. By Lemma \ref{nadirendlem} we may assume that $k,l> 0$, and by Lemma \ref{rellem} together with Proposition \ref{nadirprop} we may assume that $q$ is any prime different from $p_i$, say $q=p_{i\text{ $($mod }|P|)+1}$, obtaining the form stated in part (iv).

Next consider the case where $z$ starts on a $p_j$-edge but does not end on a $p$-edge for any $p\in P$. By parts (i) and (iii) of Lemma \ref{nadirendlem} together with Proposition \ref{nadirprop}, we may write $z$ as a linear combination of monomials of the form stated in part (v), which have no total nadir because $\res_{p_j}^k\ind_{p_1}^{l_1}\dots\ind_{p_{|P|}}^{l_{|P|}}$ is the unique nadir in $(\res_{p_1}\ind_{p_1})^t\ind_{p_j}^l\res_{p_j}^k\ind_{p_1}^{l_1}\dots\ind_{p_{|P|}}^{l_{|P|}}$ with respect to $p_j$, and this can not be a nadir with respect to every other $p_i$. 

Finally consider the case where $z$ ends on a $p_j$-edge but does not start on a $p$-edge for any $p\in P$. By parts (ii) and (iii) of Lemma \ref{nadirendlem} together with Proposition \ref{nadirprop}, we may write $z$ as a linear combination of monomials of the form stated in part (vi), which similarly to the monomials in case (v) have no total nadir. 

Now for linear independence. Let us first consider the case where every (rather than just some) simple $D_{2n}$-module $M$ in $\mathcal{M}$ satisfies that every prime factor of $n$ lies in $P$. It is clear from Lemma \ref{nadirlem} that the monomials of the forms (i)-(vi) are nonzero. From Lemma \ref{deplem} we may immediately rule out all linear dependences except for ones of the forms 
\begin{equation*}
cz=\res_{p_1}\ind_{p_1}z,
\end{equation*}
where $z$ is of one of the forms (i)-(vi) and $c\in\C$. Assume towards a contradiction that there exist some such $z$ and $c$. In particular we must have 
\begin{equation*}
(c-\res_{p_1}\ind_{p_1})zW_1(n)=0
\end{equation*}
for arbitrary $n$ satisfying $W_1(n)\in\mathcal{M}$. Note from Proposition \ref{resindprop} that $zW_1(n)$ for some $n$ chosen using Lemma \ref{nadirlem} is a nonzero linear combination of modules of the form $W_{k'}(n')$ (for some fixed $n'$ and various $k'$), and by the same proposition that $\res_{p_1}\ind_{p_1}W_{k'}(n')=p_1W_{k'}(n')$, so that we must have $c=p_1$. Similarly, we have $\res_{p_1}\ind_{p_1}(V_{a,b}(n')+V_{a,-b}(n'))=p_1(V_{a,b}(n')+V_{a,-b}(n'))$. 

Again from Proposition \ref{resindprop} we see that the property that $M\in\mathcal{M}$ with coefficients in $\N$ (with respect to the natural basis of simple dihedral group modules) has a larger $V_{1,1}(m)$-coefficient than $V_{1,-1}(m)$-coefficient for every $m$ where at least one of the coefficients is nonzero is invariant under the action of any monic monomial in $A_P$. Combining this with the above paragraph, we get that 
\begin{equation*}
(p_1-\res_{p_1}\ind_{p_1})zV_{1,1}(n)=c'(p_1-\res_{p_1}\ind_{p_1})V_{1,1}(n')
\end{equation*}
for some $c'\in\C^*$ and some $n'$. However, we may by direct computation verify that $(p_1-\res_{p_1}\ind_{p_1})V_{1,1}(n')=\frac{p_1-1}{2}(V_{1,1}(n')-V_{1,-1}(n'))\ne 0$, contradicting 
\begin{equation*}
(p_1-\res_{p_1}\ind_{p_1})zV_{1,1}(n)=0.
\end{equation*} 

Since every $z\in A_{P,\mathcal{M}}$ can be written as a linear combination of the basis elements using the relations (i)-(v) (via Proposition \ref{nadirprop}, Lemma \ref{rellem}, Corollary \ref{ccor}, and Lemma \ref{nadirendlem}), these relations indeed generate all relations of $A_{P,\mathcal{M}}$. 

Finally consider the more general case where some (but not necessarily every) simple $D_{2n}$-module $M\in\mathcal{M}$ satisfies that every prime factor of $n$ lies in $P$. Let $\mathcal{M}\cong\mathcal{N}\oplus\mathcal{N}'$, where $\mathcal{N}\subset\mathcal{M}$ is the submodule that is spanned by those $M$ that satisfy the aforementioned condition, and $\mathcal{N}'$ is spanned by those which do not. From Lemma \ref{wloglem}, we see that $A_{P,\mathcal{N}}$ has all the relations of $A_{P,\mathcal{M}}$ but potentially additional ones as well. Any such additional relation would be of the form $z=0$ where $z$ annihilates every module in $\mathcal{N}$ but not every module in $\mathcal{N}'$. However, we have by Theorem \ref{babybasisthm} a complete description of the relations of $A_{P,\mathcal{N}'}$, from which we see that every relation of $A_{P,\mathcal{N}}$ is also a relation of $A_{P,\mathcal{N}'}$, hence of $A_{P,\mathcal{M}}$. 
\end{proof}

\begin{myex}
We will exhibit examples of how monomials may be written as a linear combination of the monomials (i)-(vi) in Theorem \ref{basisthm} by going through the steps described in the proof. Let $P=\{3,5,7\}$. 

{\bf Case (i): } \begin{align*}
&\ind_3\res_5^2\res_3\ind_5=\ind_3\res_3\res_5^2\ind_5=\ind_3\res_3\res_5(\res_5\ind_5)\\&=(\res_5\ind_5)\ind_3\res_3\res_5=(\frac{5-1}{3-1}(\res_3\ind_3-3)+5)\ind_3\res_3\res_5\\&=(2\res_3\ind_3-1)\ind_3\res_3\res_5=2(\res_3\ind_3)\ind_3\res_3\res_5-\ind_3\res_3\res_5.
\end{align*}

{\bf Case (ii): } \begin{align*}
&\ind_3\res_3\ind_5\res_3\ind_3\res_5\ind_3\res_3=\ind_3(\res_3\res_3\ind_3\ind_3)\res_3\ind_5\res_5\\&=\ind_3((3+1)\res_3\ind_3-3)\res_3\ind_5\res_5\\&=4(\res_3\ind_3)\ind_3\res_3\ind_5\res_5-3\ind_3\res_3\ind_5\res_5.
\end{align*}

{\bf Case (iii): } \begin{align*}
&\res_5\res_7\res_3\ind_5\ind_7\ind_5\\&=(\res_7\ind_7)\res_3(\res_5\ind_5)\ind_5\\&=(\res_7\ind_7)(\res_5\ind_5)\res_3\ind_5\\&=(\frac{7-1}{3-1}(\res_3\ind_3-3)+7)(\frac{5-1}{3-1}(\res_3\ind_3-3)+5)\res_3\ind_5\\&=6(\res_3\res_3\ind_3\ind_3)\res_3\ind_5-7(\res_3\ind_3)\res_3\ind_5+2\res_3\ind_5\\&=6((3+1)\res_3\ind_3-3)\res_3\ind_5-7(\res_3\ind_3)\res_3\ind_5+2\res_3\ind_5\\&=17(\res_3\ind_3)\res_3\ind_5-16\res_3\ind_5.
\end{align*}

{\bf Case (iv): } \begin{align*}
&\res_7\ind_3\res_3^2\ind_7=\frac{1}{5}((5+1)\res_5\ind_5-\res_5\ind_5\res_5\ind_5)\res_7\ind_3\res_3^2\ind_7\\&=\frac{6}{5}(\res_5\ind_5)\res_7\ind_3\res_3^2\ind_7-\frac{1}{5}(\res_5\ind_5)(\res_5\ind_5)\res_7\ind_3\res_3^2\ind_7\\&=\frac{6}{5}(\res_7\ind_7)\res_5\ind_3\res_3^2\ind_5-\frac{1}{5}(\res_5\ind_5)(\res_7\ind_7)\res_5\ind_3\res_3^2\ind_5\\&=\frac{6}{5}(\frac{7-1}{3-1}(\res_3\ind_3-3)+7)\res_5\ind_3\res_3^2\ind_5\\&-\frac{1}{5}(\frac{5-1}{3-1}(\res_3\ind_3-3)+5)(\frac{7-1}{3-1}(\res_3\ind_3-3)+7)\res_5\ind_3\res_3^2\ind_5\\&=\frac{18}{5}(\res_3\ind_3)\res_5\ind_3\res_3^2\ind_5-\frac{12}{5}\res_5\ind_3\res_3^2\ind_5\\&-\frac{6}{5}(\res_3\ind_3\res_3\ind_3)\res_5\ind_3\res_3^2\ind_5+\frac{7}{5}(\res_3\ind_3)\res_5\ind_3\res_3^2\ind_5\\&-\frac{2}{5}\res_5\ind_3\res_3^2\ind_5\\&=\frac{18}{5}(\res_3\ind_3)\res_5\ind_3\res_3^2\ind_5-\frac{12}{5}\res_5\ind_3\res_3^2\ind_5\\&-\frac{6}{5}((3+1)\res_3\ind_3-3)\res_5\ind_3\res_3^2\ind_5+\frac{7}{5}(\res_3\ind_3)\res_5\ind_3\res_3^2\ind_5\\&-\frac{2}{5}\res_5\ind_3\res_3^2\ind_5\\&=\frac{1}{5}(\res_3\ind_3)\res_5\ind_3\res_3^2\ind_5+\frac{4}{5}\res_5\ind_3\res_3^2\ind_5.
\end{align*}

{\bf Case (v): } \begin{align*}
\res_3\ind_5\ind_3\res_5\ind_3\res_5=\ind_5\res_5^2(\res_3\ind_3)\ind_3=(\res_3\ind_3)\ind_5\res_5^2\ind_3.
\end{align*}

{\bf Case (vi): } \begin{align*}
\res_3\ind_5\res_3\res_5\ind_3\res_5=\res_3(\res_3\ind_3)\ind_5\res_5^2=(\res_3\ind_3)\res_3\ind_5\res_5^2.
\end{align*}
\end{myex}

\subsection{The center, and a decomposition, of \texorpdfstring{$A_{P,\mathcal{M}}$}{APM}}

\begin{mylem}
\label{idemlem}
Let $P$ be a set of odd primes and let $\mathcal{M}\subset\mathcal{G}$ be any $A_{P}$-submodule generated by $D_{2n}$-modules with $n$ odd. Let also $q\in P$ be arbitrary. Then $A_{P,\mathcal{M}}$ contains the central idempotents
\begin{equation*}
\epsilon_1=\frac{\res_q\ind_q-1}{q-1}
\end{equation*}
and
\begin{equation*}
\epsilon_2=\frac{q-\res_q\ind_q}{q-1},
\end{equation*}
which satisfy that 
\begin{equation*}
\epsilon_1+\epsilon_2=1.
\end{equation*}
\end{mylem}
\begin{proof}
That $\epsilon_1$ and $\epsilon_2$ belong to the center of $A_{P,\mathcal{M}}$ is immediate from Corollary \ref{ccor}. That they are idempotents is shown by direct calculation and an application of Lemma \ref{rellem}:
\begin{align*}
\epsilon_1^2&=(\frac{\res_q\ind_q-1}{q-1})^2=\frac{(\res_q\ind_q)^2-2\res_q\ind_q+1}{q^2-2+1}=\frac{\res_q\res_q\ind_q\ind_q-2\res_q\ind_q+1}{(q-1)^2}\\&=\frac{(q+1)\res_q\ind_q-q-2\res_q\ind_q+1}{(q-1)^2}=\frac{(q-1)(\res_q\ind_q-1)}{(q-1)^2}=\epsilon_1,
\end{align*}
and similarly for $\epsilon_2$. Also that $\epsilon_1+\epsilon_2=1$ is a result of direct computation.
\end{proof}

For a set of odd primes $P$ and $\mathcal{M}\subset\mathcal{G}$ an $A_{P}$-submodule spanned by $D_{2n}$-modules with $n$ odd and furthermore such that for each fixed $n$ either all simple $D_{2n}$-modules or none belong to $\mathcal{M}$, we define algebras 
\begin{equation*}
T^1_{P,\mathcal{M}}=A_{P,\mathcal{M}}/\langle \res_{p}\ind_{p}-p\rangle
\end{equation*}
and 
\begin{equation*}
T^2_{P,\mathcal{M}}=A_{P,\mathcal{M}}/\langle \res_{p}\ind_{p}-1\rangle,
\end{equation*}
where $p\in P$ is arbitrary. That these algebras are well-defined is part of the following theorems \ref{babydecompthm} and \ref{decompthm}. 

\begin{mythm}
\label{babydecompthm}
Let $P$ be a set of odd primes and pick an arbitrary indexing of $P$ by $\lb 1,|P|\rb$. Let also $\mathcal{M}\subset\mathcal{G}$ be some $A_P$-submodule spanned by simple $D_{2n}$-modules with $n$ odd and such that for each fixed $n$ either all or none of the simple $D_{2n}$-modules lie in $\mathcal{M}$, and furthermore such that there is no simple $D_{2n}$-module in $\mathcal{M}$ with all prime factors of $n$ belonging to $P$. Then the algebras $T^1_{P,\mathcal{M}}$ and $T^2_{P,\mathcal{M}}$ do not depend on the choice of $p$, and each has a basis consisting of the monomials of the forms
\begin{equation*}
\ind_{p_1}^{l_1}\dots\ind_{p_{|P|}}^{l_{|P|}}\res_{p_1}^{k_1}\dots\res_{p_{|P|}}^{k_{|P|}}
\end{equation*}
with $k_i,l_i\in\N$, where we have identified monomials in $A_{P,\mathcal{M}}$ with their images under the natural projections $\pi_1: A_{P,\mathcal{M}}\rightarrow T^1_{P,\mathcal{M}}$ and $\pi_2: A_{P,\mathcal{M}}\rightarrow T^2_{P,\mathcal{M}}$ respectively. Furthermore we have isomorphisms
\begin{align*}
A_{P,\mathcal{M}}\xrightarrow{\sim} A_{P,\mathcal{M}}\epsilon_1&\oplus A_{P,\mathcal{M}}\epsilon_2\xrightarrow{\sim} T^1_{P,\mathcal{M}}\oplus T^2_{P,\mathcal{M}}\\
z\mapsto z\epsilon_1&\oplus z\epsilon_2\mapsto \pi_1(z)\oplus\pi_2(z),
\end{align*}
where $\epsilon_1$ and $\epsilon_2$ depend on some fixed $q\in P$, as in Lemma \ref{idemlem}.
\end{mythm}
\begin{proof}
The proof of Theorem \ref{decompthm} applies here too, with the exception that we need to invoke Theorem \ref{babybasisthm} instead of Theorem \ref{basisthm} in it.
\end{proof}

\begin{mythm}
\label{decompthm}
Let $P$ be a set of odd primes and pick an arbitrary indexing of $P$ by $\lb 1,|P|\rb$. Let also $\mathcal{M}\subset\mathcal{G}$ be some $A_P$-submodule spanned by simple $D_{2n}$-modules with $n$ odd and such that for each fixed $n$ either all or none of the simple $D_{2n}$-modules lie in $\mathcal{M}$, and furthermore such that there is a simple $D_{2n}$-module in $\mathcal{M}$ with all prime factors of $n$ belonging to $P$. Then the algebras $T^1_{P,\mathcal{M}}$ and $T^2_{P,\mathcal{M}}$ do not depend on the choice of $p$, and each has a basis consisting of the monomials of the forms
\begin{enumerate}
\item[$($i$)$]
$\ind_{p_1}^{l_1}\dots\ind_{p_{|P|}}^{l_{|P|}}\res_{p_1}^{k_1}\dots\res_{p_{|P|}}^{k_{|P|}}$\\
with $k_i,l_i\in\N$,
\item[$($ii$)$]
$\ind_{p_1}^{l_1}\res_{p_1}^{k_1}\dots\ind_{p_{|P|}}^{l_{|P|}}\res_{p_{|P|}}^{k_{|P|}}$\\
with $k_i,l_i\in\N$ such that $k_i\ne 0\ne l_i$ for at least two $i$,
\item[$($iii$)$]
$\res_{p_i}^k\ind_{p_j}^l$\\
with $i\ne j$, and $k,l\in\Z_{>0}$,
\item[$($iv$)$]
$\res_{p_{i\text{ $($mod }|P|)+1}}\ind_{p_i}^l\res_{p_i}^k\ind_{p_{i\text{ $($mod }|P|)+1}}$\\
with $k,l\in\Z_{>0}$,
\item[$($v$)$]
$\ind_{p_j}^l\res_{p_j}^k\ind_{p_1}^{l_1}\dots\ind_{p_{|P|}}^{l_{|P|}}$\\
with $j\in\lb 1,|P|\rb$, with $k,l\in\Z_{>0}$, and $l_i\in\N$ such that $l_j=0$ but $l_i\ne 0$ for at least one $i$,
\item[$($vi$)$]
$\res_{p_1}^{k_1}\dots\res_{p_{|P|}}^{k_{|P|}}\ind_{p_j}^l\res_{p_j}^k$\\
with $j\in\lb 1,|P|\rb$, with $k,l\in\Z_{>0}$, and $k_i\in\N$ such that $k_j=0$ but $k_i\ne 0$ for at least one $i$,
\end{enumerate}
where we have identified monomials in $A_{P,\mathcal{M}}$ with their images under the natural projections $\pi_1: A_{P,\mathcal{M}}\rightarrow T^1_{P,\mathcal{M}}$ and $\pi_2: A_{P,\mathcal{M}}\rightarrow T^2_{P,\mathcal{M}}$ respectively. Furthermore we have isomorphisms
\begin{align*}
A_{P,\mathcal{M}}\xrightarrow{\sim} A_{P,\mathcal{M}}\epsilon_1&\oplus A_{P,\mathcal{M}}\epsilon_2\xrightarrow{\sim} T^1_{P,\mathcal{M}}\oplus T^2_{P,\mathcal{M}}\\
z\mapsto z\epsilon_1&\oplus z\epsilon_2\mapsto \pi_1(z)\oplus\pi_2(z),
\end{align*}
where $\epsilon_1$ and $\epsilon_2$ depend on some fixed $q\in P$, as in Lemma \ref{idemlem}.
\end{mythm}
\begin{proof}
Because $\epsilon_1$ and $\epsilon_2$ are central idempotents which add up to $1$ by Lemma \ref{idemlem}, we indeed have the isomorphism
\begin{align*}
A_{P,\mathcal{M}}&\xrightarrow{\sim} A_{P,\mathcal{M}}\epsilon_1\oplus A_{P,\mathcal{M}}\epsilon_2\\
z&\mapsto z\epsilon_1\oplus z\epsilon_2.
\end{align*}
Let us show that we have 
\begin{align*}
A_{P,\mathcal{M}}\epsilon_1&\xrightarrow{\sim} T^1_{P,\mathcal{M}}\\
z\epsilon_1&\mapsto \pi_1(z),
\end{align*}
as well as the claimed basis of $T^1_{P,\mathcal{M}}$ (the corresponding proofs for $T^2_{P,\mathcal{M}}$ are done analogously). 

In $A_{P,\mathcal{M}}$, we have
\begin{align*}
\res_p\ind_p\epsilon_1&=\res_p\ind_p\frac{\res_{q}\ind_{q}-1}{q-1}=(\frac{p-1}{q-1}(\res_{q}\ind_{q}-q)+p)\frac{\res_{q}\ind_{q}-1}{{q}-1}\\&=(p-q\frac{q}{q-1})\frac{\res_{q}\ind_{q}-1}{q-1}+\frac{p-1}{q-1}\frac{\res_{q}\res_{q}\ind_{q}\ind_{q}-\res_{q}\ind_{q}}{{q}-1}\\&=(p-{q}\frac{p-1}{{q}-1})\frac{\res_{q}\ind_{q}-1}{{q}-1}+\frac{q}{{q}-1}\frac{({q}+1)\res_{q}\ind_{q}-{q}-\res_{q}\ind_{q}}{{q}-1}\\&=(p-{q}\frac{p-1}{{q}-1})\frac{\res_{q}\ind_{q}-1}{{q}-1}+{q}\frac{p-1}{{q}-1}\frac{\res_{q}\ind_{q}-1}{{q}-1}=p\epsilon_1.
\end{align*}
Thus we have a natural epimorphism
\begin{align*}
A_{P,\mathcal{M}}/\langle \res_{p}\ind_{p}-p\rangle &\twoheadrightarrow A_{P,\mathcal{M}}\epsilon_1\\
1&\mapsto\epsilon_1.
\end{align*}
From this epimorphism and the basis of $A_{P,\mathcal{M}}$ given in Theorem \ref{basisthm} we get that the elements of the forms 
\begin{align*}
z\epsilon_1=\frac{1}{q-1}(\res_{q}\ind_{q}z-z)
\end{align*}
with $z$ of the form (i)-(vi) span $A_{P,\mathcal{M}}\epsilon_1$, and by the same theorem that they are even linearly independent. 
\end{proof}

\begin{myprop}
\label{tisoprop}
For a set of odd primes $P$ and $\mathcal{M}\subset\mathcal{G}$ an $A_{P}$-submodule spanned by $D_{2n}$-modules with $n$ odd and furthermore such that for each fixed $n$ either all simple $D_{2n}$-modules or none belong to $\mathcal{M}$, we have a mapping
\begin{align*}
T^1_{P,\mathcal{M}}&\rightarrow T^2_{P,\mathcal{M}}\\
\res_p&\mapsto \res_p\\
\ind_p&\mapsto p\ind_p
\end{align*}
which extends to an isomorphism of algebras. 
\end{myprop}
\begin{proof}
The relation $\res_p\ind_p=p$ in $T^1_{P,\mathcal{M}}$ is preserved by the mapping because of the relation $\res_p\ind_p=1$ in $T^2_{P,\mathcal{M}}$. The other relations (relations (i)-(v) as given in Theorems \ref{babybasisthm} and \ref{basisthm} respectively) are preserved because they either are special cases of the previous relation or are homogeneous in $\ind_p$. 
\end{proof}
Let $B=\langle a,b|ab=1\rangle$ be the bicyclic algebra.
\begin{mycor}
\label{bicyccor}
Let $P$ be a set of odd primes. Let also $\mathcal{M}\subset\mathcal{G}$ be some $A_P$-submodule spanned by simple $D_{2n}$-modules with $n$ odd and such that for each fixed $n$ either all or none of the simple $D_{2n}$-modules lie in $\mathcal{M}$, and furthermore such that there is no simple $D_{2n}$-module in $\mathcal{M}$ with all prime factors of $n$ belonging to $P$. Then
\begin{equation*}
A_{P,\mathcal{M}}\cong (\bigotimes_{p\in P}B)^2.
\end{equation*}
\end{mycor}
\begin{proof}
Thanks to Theorem \ref{babydecompthm} and Proposition \ref{tisoprop}, it suffices to show that
\begin{equation*}
T^2_{P,\mathcal{M}}\cong \bigotimes_{p\in P}B.
\end{equation*}
Let, for every $p\in P$, the algebra $B_p=\langle a_p,b_p|a_pb_p=1\rangle$ be a copy of the bicyclic algebra. Then, considering the basis of $T^2_{P,\mathcal{M}}$ given in Theorem \ref{babydecompthm} as well as the relations $\res_p\ind_p=1$ in $T^2_{P,\mathcal{M}}$ and relations (i)-(iii) of Theorem \ref{babybasisthm}, we clearly have an isomorphism defined by
\begin{align*}
T^2_{P,\mathcal{M}}&\xrightarrow{\sim} \bigotimes_{p\in P}B_p\\
\res_p&\mapsto a_p\\
\ind_p&\mapsto b_p.
\end {align*}
\end{proof}

\begin{mylem}
\label{linindlem}
With the setup of Theorem \ref{decompthm}, for any two different monomials $z_1$ and $z_2$ either both of the form (i) or both of one of the forms (ii)-(vi) from that theorem, there exists some $q\in P$ such that the respective termini or the respective nadirs of $z_1$ and $z_2$ with respect to $q$ are different.
\end{mylem}
\begin{proof}
It is clear from considering whether $z$ starts or ends on various $p$-edges in the proof of Theorem \ref{basisthm} that the lemma statement holds for $z_1$ and $z_2$ being of different forms (ii)-(vi). For each fixed form (i)-(vi) it is straightforward to verify that the termini and nadirs with respect to the $p\in P$ uniquely determine a monomial of that form. 
\end{proof}

\begin{mylem}
\label{clem}
Let $P$ be a set of odd primes, and let $\mathcal{M}\subset\mathcal{G}$ be an $A_{P}$-submodule spanned by $D_{2n}$-modules with $n$ odd and furthermore such that for each fixed $n$ either all simple $D_{2n}$-modules or none belong to $\mathcal{M}$. Then the only central elements of $T^2_{P,\mathcal{M}}$ which are linear combinations of monomials which have a total nadir are the scalars. 
\end{mylem}
\begin{proof}
Assume towards a contradiction that $z\in T^2_{P,\mathcal{M}}\backslash\C$ lies in the center and is a linear combination of monomials each having a total nadir. Then $z$ can be written as a linear combination of monomials of the form (i) as in Theorem \ref{decompthm}. First consider the case where $z$ contains some factor $\res_p$ for some fixed $p\in P$. Let $z=z_1+z_2$ where the monomial terms of $z_1$ contain a factor $\res_p$ while the monomial terms of $z_2$ do not. Then 
\begin{equation*}
\ind_pz-z\ind_p=\ind_pz_1-z_1\ind_p.
\end{equation*}
Multiplying a monomial term of $z_1$ by $\ind_p$ from the right increases both the nadir and the terminus of the monomial with respect to $p$ by 1, while leaving the other termini and nadirs unchanged. It then follows from Lemmata \ref{nadirlem} and \ref{linindlem} that this multiplication does not annihilate any monomial terms of $z_1$. The largest nadir of any term of $\ind_pz_1-z_1\ind_p$ with respect to $p$ is clearly to be found in $z_1\ind_p$ and not in $\ind_pz_1$. But then $\ind_pz-z\ind_p$ will contain a nonzero multiple of such a term, contradicting $\ind_pz-z\ind_p=0$. The case where the monomial terms of $z$ only have factors $\ind_p$ for various $p\in P$ is handled by fixing one such $p\in P$, considering the expression
\begin{equation*}
\res_p z-z\res_p,
\end{equation*}
and applying a similar argument. 
\end{proof}

\begin{mythm}
\label{cthm}
Let $P$ be a set of odd primes, let $p\in P$ be arbitrary, and let $\mathcal{M}\subset\mathcal{G}$ be an $A_{P}$-submodule spanned by $D_{2n}$-modules with $n$ odd and furthermore such that for each fixed $n$ either all simple $D_{2n}$-modules or none belong to $\mathcal{M}$. Then the center of $A_{P,\mathcal{M}}$ is generated by 1 and $\res_p\ind_p$. 
\end{mythm}
\begin{proof}
By Theorems \ref{babydecompthm} and \ref{decompthm}, together with Proposition \ref{tisoprop}, it suffices to show that the center of $T^2_{P,\mathcal{M}}$ is $\C$. Assume towards a contradiction that $z$ lies in the center of $T^2_{P,\mathcal{M}}$ but not in $\C$. Let $l$ be the degree of $z$ (as a polynomial in various $\res$ and $\ind$).

For arbitrary monic monomials $z',z''\in T^2_{P,\mathcal{M}}$, let $e'_q$ be the terminus of $z'$ and $d'_q$ the nadir of $z'$, and also $e''_q$ the terminus of $z''$ and $d''_q$ the nadir of $z''$, with respect to $q$ for all $q\in P$. Note that $z'z''$ and $z''z'$ both have terminus $e'_q+e''_q$ with respect to $q$, and that by assumption we have the relation
\begin{equation*}
z'z=zz'
\end{equation*}
in $T^2_{P,\mathcal{M}}$. Using Lemma \ref{deplem} and the fact that the monomials involved in the extra relation $\res_q\ind_q=1$ all have the same termini, we may assume that all monomial terms in $z$ have the same termini (with respect to every $q\in P$). 

Using (for instance) the same indexation of the primes in $P$ as in Theorem \ref{basisthm}, let 
\begin{equation*}
x=\res_{p_1}^l\res_{p_2}^l\dots\res_{p_{|P|}}^l. 
\end{equation*} 
Let furthermore $z=z_1+z_2$, where the monomial terms of $z_1$ have a total nadir while the monomial terms of $z_2$ do not. Then $z_1$ can be written as a linear combination of basis elements of the form (i) in Theorem \ref{decompthm}. 

Note that $xz'$ is a total nadir in $xz'$ if $z'$ has degree at most $l$, because for each $q\in P$, the number of factors $\res_q$ in $x$ is greater than or equal to the number of factors $\ind_q$ in $z'$. Also, $z''x$ is a nadir in $z'x$ with respect to $q$ if and only if $z''$ is a nadir in $z'$ with respect to $q$. In particular, $xz_1$, $z_1x$, and $xz_2$ will all have a total nadir, while $z_2x$ will not. Write these expressions in the basis of Theorem \ref{decompthm} to see that it then follows from $xz-zx=0$ that $z_2x=0$

Now, either $z_2=0$ (which is in particular the case if there is no $D_{2n}$-module in $\mathcal{M}$ with all prime factors of $n$ belonging to $P$, as in Theorem \ref{babybasisthm}) or $z_2\ne 0$. In the first case, we are done by Lemma \ref{clem}. In the second case, let $u$ be an arbitrary monomial term of $z_2$ when the latter is expressed in the basis given in Theorem \ref{decompthm}. Let $e_q$ be the terminus of $u$ with respect to $q\in P$ and $d_q$ be the nadir of $u$ with respect to $q$. Using (for instance) Lemma \ref{nadirlem} we see that $ux\ne 0$. Hence the relation $z_2x=0$ is nontrivial. 

Moreover, the terminus of $ux$ with respect to $q$ is $e_q-l$, and the nadir of $ux$ with respect to $q$ is $d_q-l$. In particular, we get by Lemma \ref{linindlem} that the linear independence of the terms of $z_2$ (expressed in the basis of Theorem \ref{decompthm}) is preserved by right multiplication by $x$. But that $z_2x$ is a linear combination of linearly independent elements where not all coefficients are zero contradicts $z_2x=0$. 
\end{proof}

\begin{mycor}
\label{indeccor}
Let $P$ be a set of odd primes, and let $\mathcal{M}\subset\mathcal{G}$ be an $A_{P}$-submodule spanned by $D_{2n}$-modules with $n$ odd and furthermore such that for each fixed $n$ either all simple $D_{2n}$-modules or none belong to $\mathcal{M}$. Then the algebras $T^1_{P,\mathcal{M}}$ and $T^2_{P,\mathcal{M}}$ are indecomposable.
\end{mycor}
\begin{proof}
If $T^1_{P,\mathcal{M}}$ or $T^2_{P,\mathcal{M}}$ were decomposable, then the identity of any summand would be a non-scalar central element of the sum. This element would via theorems \ref{babydecompthm} and \ref{decompthm} correspond to a central element of $A_{P,\mathcal{M}}$ not generated by $1$ and $\res_p\ind_p$ for any $p\in P$, which would contradict Theorem \ref{cthm}. 
\end{proof}

\section{Results for restriction and induction with respect to the prime 2, and further directions}\label{s6}
This section deals with some partial results for the case where $2\in P$, and seeks to point towards some directions suitable for further investigation. The discussion will largely be intended to convey intuition, at some expense to rigor. 

Induction and restriction of $D_{2n}$-modules for $2\in P$ (and hence possibly even $n$) handles differently than for the odd primes only, as is evident from Proposition \ref{resindprop} (and also from Figures 1 and 2). This leads to the failure of analogues of some of the results leading up to the main result of the preceding section, Theorem \ref{basisthm}. We recover the following analogue of Lemma \ref{rellem}.

\begin{myprop}
\label{relprop}
Let $P\ni 2$ be a set of prime numbers, and let $\mathcal{M}\subset\mathcal{G}$ be an $A_P$-submodule. Then
\begin{equation*}
\varphi_{P,\mathcal{M}}(\res_2\res_2\res_2\ind_2\ind_2\ind_2)=\varphi_{P,\mathcal{M}}(3\res_2\res_2\ind_2\ind_2-2\res_2\ind_2).
\end{equation*}
\end{myprop}
\begin{proof}
By direct computation using Proposition \ref{resindprop} (or by looking at Figure 2) we have for $n\ge 3$ that
\begin{align*}
&(2\res_2\ind_2-3\res_2\res_2\ind_2\ind_2+\res_2\res_2\res_2\ind_2\ind_2\ind_2)W_k(n)\\&=2\cdot 2W_k(n)-3\cdot 2^2W_k(n)+2^3W_k(n)=0,
\end{align*}
and
\begin{align*}
&(2\res_2\ind_2-3\res_2\res_2\ind_2\ind_2+\res_2\res_2\res_2\ind_2\ind_2\ind_2)V_{1,b}(n)\\&=2\cdot 2V_{1,b}(n)-3(3V_{1,b}(n)+V_{1,-b}(n))+5V_{1,b}(n)+3V_{1,-b}(n)=0,
\end{align*}
and
\begin{align*}
&(2\res_2\ind_2-3\res_2\res_2\ind_2\ind_2+\res_2\res_2\res_2\ind_2\ind_2\ind_2)V_{-1,b}(n)\\&=2(V_{-1,b}(n)+V_{-1,-b}(n))-3(2V_{-1,b}(n)+2V_{-1,-b}(n))+4V_{-1,b}(n)+4_{-1,-b}(n)=0.
\end{align*}
The desired result follows.
\end{proof}
The above relation was indeed discovered in the same way as the relation of Lemma \ref{rellem}: For each $p\in P$, the actions of the monomials of the form $\res_p^l\ind_p^l$ on $\mathcal{M}$ is, because of the regularity of the induction/restriction diagrams, determined by their actions on a finite number of simple modules in $\mathcal{M}$, and on the span of these modules, the infinitely many monomials $\res_p^l\ind_p^l$ act as endomorphisms, hence only a finite number of them can be linearly independent. It is now not difficult to find the above linear dependence explicitly.

Essential for Theorem \ref{basisthm} were also our ability to commute factors corresponding to different primes using Proposition \ref{nadirprop}, the linear independence of monomials of different nadirs and ``total nadirity'' by Lemma \ref{deplem}, and the simplification provided by each $\res_p\ind_p$ being central in $A_{P,\mathcal{M}}$ by Corollary \ref{ccor}. The first two of these results are a consequence of Lemma \ref{partfunlem}, which in turn relies on properties of the induction/restriction diagrams discussed in the paragraphs preceding the lemma. By inspection of the induction/restriction diagram for the prime $2$ (see Figure 2), it seems likely that we by a construction similar to that of Lemma \ref{partfunlem} may perform the ``upward translation'' required for Proposition \ref{nadirprop}. The ``downward translation'' used for Lemma \ref{deplem} seems to fail for the prime $2$, since either the module $W_m(4m)$ or the modules $V_{-1,b}(2m)$ (depending on the parity of $m$) can not be translated to a module at the bottom level of the diagram. Nevertheless, it seems very plausible to me that this quite small gap in the proof of an analogue to Lemma \ref{deplem} may be bridged by other means. Here one may mention one additional case which has so far been swept under the rug, namely the case where $2\not\in P$ but where $\mathcal{M}$ contains some $D_{2n}$-module with $n$ even. Here, like above, the downward translation used for Lemma \ref{deplem} seems to fail, but my guess is that one could resolve this issue without much trouble and arrive at results similar to the Theorem \ref{babybasisthm} case. 

As for hopes of finding an analogue to Corollary \ref{ccor}, it is easily verified (e.g. by computing $\res_2\res_2\ind_2V_{-1,1}(2m)\ne \res_2\ind_2\res_2V_{-1,1}(2m)$) that $\res_2\ind_2$ is not central in $A_{P,\mathcal{M}}$. Instead, we have the following relations.
\begin{myprop}
Let $P\ni 2$ be a set of prime numbers, and let $\mathcal{M}\subset\mathcal{G}$ be an $A_P$-submodule. Then
\begin{enumerate}
\item[$($i$)$]
$\varphi_{P,\mathcal{M}}(\ind_2\res_2\ind_2)=\varphi_{P,\mathcal{M}}(2\ind_2)$,

\item[$($ii$)$]
$\varphi_{P,\mathcal{M}}(\res_2\ind_2\res_2)=\varphi_{P,\mathcal{M}}(2\res_2)$.
\end{enumerate}
\end{myprop}
\begin{proof}
This is done by straightforward computation using Proposition \ref{resindprop} (or by looking at Figure 2) similarly to the proof of Proposition \ref{relprop}.
\end{proof}
We also have the following relation.
\begin{myprop}
Let $P\ni 2$ be a set of prime numbers, and let $\mathcal{M}\subset\mathcal{G}$ be an $A_P$-submodule. Then
\begin{equation*}
\varphi_{P,\mathcal{M}}((\res_2^2\ind_2^2)^4)=\varphi_{P,\mathcal{M}}(6(\res_2^2\ind_2^2)^3-8(\res_2^2\ind_2^2)^2).
\end{equation*}
\end{myprop}
\begin{proof}
This is again done by straightforward computation using Proposition \ref{resindprop} (or by looking at Figure 2) similarly to the proof of Proposition \ref{relprop}.
\end{proof}

There may well exist additional relations, but the problems of finding these and ultimately a basis for $A_{P,\mathcal{M}}$ are likely more difficult to solve than for the case of odd primes, although possibly within reasonable reach for future investigation.

There are additional directions which would be natural to pursue on the topic of the algebras $A_{P,\mathcal{M}}$. One is that of the cases where $\mathcal{M}$ does not necessarily contain either all or none of the simple $D_{2n}$-modules for a fixed $n$. As mentioned in the beginning of Section \ref{s5}, this should not be too difficult. Another is to go the Coxeter route when defining the dihedral groups, and in particular obtain a well-defined group $D_{2n}$ also for $n=1,2$. This would remove the relevance of total nadirs in our proof, giving us no reason to distinguish between the cases of Theorem \ref{babybasisthm} and Theorem \ref{basisthm} respectively. I expect the outcome to be very similar to the Theorem \ref{babybasisthm} case. In all of these cases, it may be interesting to look into the representation theory of the algebras $A_{P,\mathcal{M}}$. 

Finally, we note that a more general class of algebras corresponding to induction/restriction diagrams of sufficient regularity should be amenable to methods used in this paper. Indeed, diagrams satisfying the property described in the discussion preceding Lemma \ref{partfunlem} should as noted above admit analogues to Proposition \ref{nadirprop} and Lemma \ref{nadirlem}, and also an analogue to Lemma \ref{rellem} and Proposition \ref{relprop} by the discussion succeeding the latter.

\end{document}